\documentclass{article}

\usepackage[shortlabels]{enumitem}
\usepackage{amsfonts,amsmath, soul, matlab-prettifier, bm, bbm, amsthm,enumitem,amssymb,multirow,float,mathtools,bbm,array,varwidth,hyperref,bm}
\usepackage[all]{xy}
\usepackage[margin=0.9in]{geometry}
\usepackage{graphicx}
\usepackage{mathtools, caption, dsfont}
\usepackage{tikz}
\usepackage{amsmath}
\usetikzlibrary{arrows.meta, decorations.pathmorphing, calc}
\numberwithin{equation}{subsection}

\mathchardef\mhyphen="2D

\newcommand{\Hom}{\mathrm{Hom}}

\newcommand{\GL}{\mathrm{GL}}
\newcommand{\SL}{\mathrm{SL}}

\newcommand{\ind}{\mathrm{ind}}

\newcommand{\supp}{\mathrm{supp}}
\newcommand{\End}{\mathrm{End}}

\newcommand{\val}{\mathrm{val}}

\newcommand{\cO}{\mathcal{O}}
\newcommand{\cF}{\mathcal{F}}

\newcommand{\cH}{\mathcal{H}}

\newcommand{\cW}{\mathcal{W}}
\newcommand{\cA}{\mathcal{A}}
\newcommand{\cI}{\mathcal{I}}

\newcommand{\cV}{\mathcal{V}}
\newcommand{\cC}{\mathcal{C}}
\newcommand{\cG}{\mathcal{G}}

\newcommand{\cB}{\mathcal{B}}
\newcommand{\cS}{\mathcal{S}}

\newcommand{\CC}{\mathbb{C}}
\newcommand{\FF}{\mathbb{F}}

\newcommand{\RR}{\mathbb{R}}
\newcommand{\ZZ}{\mathbb{Z}}
\newcommand{\GG}{\mathbb{G}}

\newcommand{\TT}{\mathbb{T}}

\newcommand{\fg}{\mathfrak{g}}

\newcommand{\pair}[2]{\langle #1,#2 \rangle}

\newcommand{\der}{\mathrm{der}}
\newcommand{\cpder}{(G(k)_{x,+:1})_\der}
\newcommand{\ab}{\mathrm{ab}}
\newcommand{\ur}{\mathrm{ur}}

\theoremstyle{plain}
\newtheorem{thm}{Theorem}[subsection]

\newtheorem{prop}[thm]{Proposition}

\newtheorem{lemma}[thm]{Lemma}
\newtheorem{cor}[thm]{Corollary}

\theoremstyle{definition}
\newtheorem{defn}[thm]{Definition}
\newtheorem{notn}[thm]{Notation}
\newtheorem{rem}[thm]{Remark}
\newtheorem{example}[thm]{Example}

\makeatletter
\let\c@equation\c@thm
\makeatother

\title{Exceptional supercuspidal representations in small residue characteristic}
\author{Yiannis Fam}

\begin{document}
	\maketitle
	\thispagestyle{empty}
	\pagenumbering{arabic}
	\thispagestyle{empty}

	\setcounter{page}{1}

	\begin{abstract}
		In this paper, in residue characteristic 2 and 3, we extend the construction of epipelagic representations of Reeder--Yu to produce new supercuspidals of higher depth, building on work of Gastineau. In particular, we produce examples of epipelagic representations that do not arise from the construction of Reeder--Yu.
	\end{abstract}

\section{Introduction}
Let $k$ be a non-archimedean local field with residue characteristic $p$, residue field $\FF_q$, ring of integers $\cO_k$ and uniformiser $\varpi$. Let $G$ be a connected reductive group over $k$, and assume that $G$ splits over a tamely ramified extension of $k$. By a representation of $G(k)$ we mean a smooth complex representation. This paper focuses on the construction of supercuspidal representations of $G(k)$ when $p$ is small. 

The construction of supercuspidal representations of $\GL_n(k)$ and $\SL_n(k)$ was completed by Bushnell--Kutzko in \cite{BK93} and \cite{BK94}. The case of inner forms of $\GL_n(k)$ was completed by S{\'e}cherre--Stevens \cite{SS08}. In \cite{Ste08}, Stevens constructs all supercuspidal representations when $p>2$ and $G$ is a classical group. For general $G$, by work of Kim \cite{Kim07} and Fintzen \cite{Fin21}, Yu's construction \cite{Yu01} of supercuspidal representations is exhaustive when $p$ is sufficiently large; \cite{Fin21} shows that it suffices to take $p$ not dividing the order of the Weyl group of $G$. Yu's construction exists when $p \neq 2$, and has recently been extended to the $p=2$ case by \cite{FS25}. One is then interested in constructing, when $p$ is small, irreducible supercuspidal representations of $G(k)$ that are not expected to arise from Yu's construction. The cases when $p=2$ or $G$ is an exceptional group are particularly interesting.

The case of depth-zero supercuspidal representations has been completed by Moy--Prasad \cite{MP1} \cite{MP2}, and separately by Morris \cite{Mor99}. In positive depth there exist constructions, uniform in the residue characteristic $p$, due to Gross--Reeder \cite{GR10} and Reeder--Yu \cite{RY14} of supercuspidal representations that do not necessarily arise from the construction of Yu (or \cite{FS25}). For the rest of this paper we assume that $G$ is a simple split reductive group over $k$ (although \cite{RY14} works in greater generality). We will, for convenience, label the first three terms of the Moy--Prasad filtration of $G(k)$ at a point $x$ in the building $\cB(G,k)$ of $G$ by $G(k)_{x,0}$, $G(k)_{x,+}$ and $G(k)_{x,++}$. Let $\GG_x/\FF_q$ denote the reductive quotient of $G$ at $x$ so that $G(k)_{x,0}/G(k)_{x,0+} \cong \GG_x(\FF_q)$.

The epipelagic representations constructed in \cite{RY14} take as input a point $x \in \cB(G,k)$, a linear functional $\lambda$ on the $\FF_q$-vector space $G(k)_{x,+}/G(k)_{x,++}$, and a fixed nontrivial additive character $\chi: \FF_q \to \CC^\times$. When $\lambda$ is stable under the action of $\GG_x$ in the sense of Geometric Invariant Theory \cite{Mum77}, the compactly induced representation $\ind_{G(k)_{x,+}}^{G(k)} \chi \circ \lambda$ is a finite direct sum of irreducible supercuspidal representations of $G(k)$. These irreducible constituents are examples of epipelagic representations, and may also be realised as the compact induction of a representation of a slightly larger compact open subgroup than $G(k)_{x,+}$. When specialising to the case that $x$ is the barycentre of an alcove, one recovers the simple supercuspidal representations of \cite{GR10}. The corresponding condition on $\lambda$ in this setting is more elementary to describe. Outside of this case the stable functionals $\lambda$ have not been classified in general. When $p$ is large, the stability condition may be transferred to a stability condition in Vinberg--Levy theory \cite{RY14} where \cite{RLYG12} applies. The assumption on $p$ was later removed by \cite{FR17}. The points $x$ for which $G(k)_{x,+}/G(k)_{x,++}$ admits stable functionals have been classified in \cite[Corollary 5.1]{RY14}. Reeder--Yu find that these stability points are rather rare; when $G=\SL_n$ the only stable point is the barycentre of an alcove.

In Section 2, to extend the above construction to higher depth we consider the compact induction of characters of $G(k)_{x,+}$ that do not factor through $G(k)_{x,+}/G(k)_{x,++}$. In order to do so, we must explicitly describe the abelianisation of $G(k)_{x,+}$. We show that this abelianisation is a quotient of $G(k)_{x,+}/G(k)_{x,1+}$ when $p>2$. To make calculation more manageable, we only work with the smaller quotient $G(k)_{x,+}/G(k)_{x,1}$. In this section we describe an algorithm that computes the abelianisation $V(x)$ of $G(k)_{x,+}/G(k)_{x,1}$ when $x$ is the barycentre of a facet. When $p>3$, we show that this abelianisation is $G(k)_{x,+}/G(k)_{x,++}$ and we are constrained to the setting of \cite{RY14}. When $p=2$ or $3$ we find that $V(x)$ may be larger than $G(k)_{x,+}/G(k)_{x,++}$, but still always admits the structure of an $\FF_q$-vector space. The question of describing such characters of $G(k)_{x,+}$ was previously considered by \cite{Gas20}. Taking $q=2$, $G=\mathrm{Sp}_4$ and $x$ the barycentre of an alcove, Gastineau produces a character $\rho: \mathrm{Sp}_4(k)_{x,+} \to \CC^\times$ that does not factor through $\mathrm{Sp}_4(k)_{x,+}/\mathrm{Sp}_4(k)_{x,++}$, such that $\ind_{\mathrm{Sp}_4(k)_{x,+}}^{\mathrm{Sp}_4(k)} \rho$ is an irreducible supercuspidal representation.

We now describe our main result. In Section 3, we review the construction of supercuspidal representations due to \cite{RY14} and verify that they are epipelagic. In Definition \ref{def:hm} we define a weaker notion of stability than the notion from Geometric Invariant Theory, that of $\FF_q$-stability. In short, this asks for the Hilbert--Mumford criterion (\cite{Mum77}) to hold for the action of any one parameter family $\GG_m \to \GG_x$, defined over $\FF_q$, on $V(x)^\vee$. We prove the following:

\begin{thm}[Theorem \ref{thm:main}]
    Fix a nontrivial additive character $\chi: \FF_q \to \CC^\times$. Let $\lambda \in V(x)^\vee$ be an $\FF_q$-stable functional and let $\chi_\lambda$ denote the inflation of $\chi \circ \lambda$ to a character of $G(k)_{x,+}$. The compactly induced representation $$\pi_\lambda = \ind_{G(k)_{x,+}}^{G(k)} \chi_\lambda$$ is a finite direct sum of irreducible supercuspidal representations of $G(k)$.
\end{thm}

In contrast with \cite{RY14}, our method does not generally give us an explicit description of the irreducible constituents of $\pi_\lambda$ as compactly induced representations.

In Section 4, we give examples of supercuspidal representations using Theorem \ref{thm:main}. These examples have the following features:
\begin{itemize}
    \item For $G=G_2$, $q=2$ and $x$ the barycentre of an alcove, we produce two irreducible supercuspidal representations of $G(k)$. We exhibit an isomorphism between these and epipelagic supercuspidal representations constructed by \cite[Section 7.5]{RY14} from a \textit{different} point in the building.
    %\item For $G=F_4$ and $p=2$ we produce a representation $\pi_\lambda$ that is a finite direct sum of irreducible supercuspidals of $G(k)$. Each summand has depth at least $1/8$ and at least one is epipelagic of depth $1/8$. 
    \item For $G$ of type $B_n$, $q=2$ and $x$ the barycentre of an alcove, we produce irreducible supercuspidal representations of $G(k)$, of which at least one does not arise from the construction of \cite{RY14} by depth considerations. However, we show how these may be obtained from the setting of Reeder--Yu using our weaker notion of $\FF_q$-stability.
    \item For $G=G_2$ and $p=3$ we produce an irreducible supercuspidal representation of $G(k)$ that is epipelagic of depth $1/2$ and does not arise from the construction of \cite{RY14} by formal degree considerations.
\end{itemize}

In particular, the latter two give examples of epipelagic representations that do not arise from the construction of \cite{RY14}.
% the existence of which was not previously known to the knowledge of the author.

\subsection{Acknowledgements}

It is a pleasure to thank my supervisor Beth Romano for her invaluable support and guidance in the writing of this paper and shaping of this project. I would also like to thank Ana Caraiani, Jessica Fintzen, Wee Teck Gan and Shaun Stevens for helpful comments, corrections and questions. This work was supported by
the Engineering and Physical Sciences Research Council [EP/S021590/1], the EPSRC Centre
for Doctoral Training in Geometry and Number Theory (The London School of Geometry and
Number Theory) at Imperial College London.

\subsection{Notation and Conventions}

Let $r+$ be a set of symbols indexed by $r \in \RR$ and let $\tilde{\RR}$ denote $\RR \sqcup \{r+ : r \in \RR\}$ with a total ordering defined by the following property: if $r<s$ are real numbers then $$r<r+<s<s+.$$

We write the commutator of two elements $a,b$ of a group to be $[a,b]=aba^{-1}b^{-1}$.

\section{Characters of some compact open subgroups}
In this section we classify characters of certain compact open subgroups of $G(k)$, specifically the first positively indexed subgroup in the Moy--Prasad filtration of $G(k)$ corresponding to a barycentre of a facet in the Bruhat--Tits building of $G(k)$.

\subsection{Notation and review of Bruhat--Tits theory}

Let $k$ be a non-archimedean local field with residue characteristic $p$, residue field $\FF_q$, ring of integers $\cO_k$ and uniformiser $\varpi$. Let $k^\ur$ be the maximal unramified extension of $k$ in a fixed algebraic closure $\overline{k}$. Let $\cO$ denote the ring of integers of $k^\ur$ and identify the residue field with $\overline \FF_q$.

Let $G$ be a simple split reductive group over $k$ with split $k$-rational maximal torus $T$. Denote the centre of $G$ by $Z$. Let $X_*(T) = \Hom_k(\GG_m,T)$ and $X^*(T) = \Hom_k(T,\GG_m)$ be the cocharacter and character lattices respectively of $T$. Let $\langle,\rangle: X^*(T) \times X_*(T) \to \ZZ$ be the natural pairing. We reuse this notation for the pairing $X^*(T)  \times X_*(T) \otimes \RR \to \RR$. Let $\Phi$ be the set of roots with respect to $T$, and for each $\alpha \in \Phi$ let $U_\alpha$ denote the corresponding root subgroup of $G$. Fix isomorphisms $u_\alpha: \GG_a \xrightarrow{\sim} U_\alpha$, defining a Chevalley basis $e_\alpha = du_\alpha(1) \in \mathrm{Lie}(U_\alpha)$ of $\fg := \mathrm{Lie}(G)$.

Let $\cA= \cA(T,k)$ denote the apartment of $T$ over $k$ and $\cB = \cB(G,k)$ denote the Bruhat--Tits building of $G$ over $k$. Our choice of Chevalley basis determines a unique point in $\cA$ fixed by $u_\alpha(-1)u_{-\alpha}(1)u_\alpha(-1)$ for all $\alpha \in \Phi$. Taking this hyperspecial point $x_0$ as the origin, we identify $\cA$ with $X_*(T) \otimes \RR$.

Let $\Psi$ be the set of affine roots with respect to $T$. These are the affine functions on $X_*(T) \otimes \RR$ of the form $\alpha+n$ for $\alpha$ a root and $n$ an integer. The gradient of an affine root $\psi =\alpha+n$ is defined to be $\alpha \in \Phi$ and we denote this by $\dot\psi$. For $\psi=\alpha+n$ let the affine root group $U_\psi$ denote $u_\alpha(\varpi^n \cO_k) \leq U_\alpha$. We have an isomorphism $u_\psi: \cO_k \xrightarrow{\sim} U_\psi$ of $\cO_k$-modules determined by sending 1 to $u_\alpha(\varpi^n)$. For any affine root $\psi$, denote by $V_\psi$ the quotient $U_\psi/U_{\psi+1}$ isomorphic to $\FF_q$ as an abelian group. Let $\overline{u}_\psi: \cO_k/\varpi \cO_k \xrightarrow{\sim} V_\psi$ be the isomorphism induced by $u_\psi$.

Let $W(\Phi)$ and $W(\Psi)$ denote the Weyl and affine Weyl groups respectively. As $T$ is split, we may view it as the generic fibre of a split torus $T_{\cO_k}$ over $\cO_k$. Denote by $N$ the normaliser of $T$ in $G$, and define the Iwahori--Weyl group $W(G)$ by $W(G)=N(k)/T_{\cO_k}(\cO_k)$. This contains $W(\Psi)$ as a finite-index normal subgroup, and the quotient $\Omega$ is isomorphic to the centre of the dual group $\hat G$.

Fix a choice of root basis for $\Phi$ and hence a choice of positive roots $\Phi^+$. Let $\cC$ be the alcove (maximal facet) in $\cA$ determined by the equations $0<\alpha(x) <1$ for all $\alpha \in \Phi^+$. The closure of $\cC$ contains the hyperspecial point $x_0$. Let $\Pi$ denote the set of simple affine roots with respect to our choice of root basis. Let $\Psi^+$ denote the positive affine roots with respect to $\Pi$. These are the affine roots that take positive values on $\cC$ and are non-negative integral linear combinations of affine roots in $\Pi$. 
%We may view the Weyl group $W(\Phi)$ as a subgroup of $W(G)$ by identifying it with the stabiliser of the hyperspecial vertex.

Fix a point $x \in \cA$ that is the barycentre of a facet in the closure of $\cC$. Let $G(k)_x$ denote the stabiliser in $G(k)$ of $x \in\cB$. This is a compact open subgroup of $G(k)$. For any non-negative real number $r$, we define the Moy--Prasad filtration subgroups 
$$G(k)_{x,r} := \langle T(F)_r, \hspace{0.5em} U_\psi : \psi(x) \geq r \rangle,$$
where
$$T(F)_0 = \{t \in T(F) : \val(\chi(t)) = 0 \text{ for all } \chi \in X^*(T)\} = T_{\cO_k}(\cO_k)$$
and, when $r>0$, 
$$T(F)_r = \{t \in T(F)_0 : \val(\chi(t)-1) \geq r \text{ for all } \chi \in X^*(T)\}.$$
Define  
$$G(k)_{x,r+} := \bigcup\limits_{s>r} G(k)_{x,s}.$$
%At any point $x$, the Moy--Prasad filtration $G(k)_{x,r}$ forms a neighbourhood basis of the identity (because any open subgroup of $G(k)$ will contain $T_s$ and $U_\alpha(\varpi^s \cO_k)$ for all $\alpha \in \Phi$, for sufficiently large $s$). The subgroup $G(k)_{x,0}$ is known as a parahoric subgroup of $G(k)$, and is finite index and normal in $G(k)_x$. 
The jumps in the Moy--Prasad filtration occur at a discrete set of real numbers. We often write $G(k)_{x,+}$ and $G(k)_{x,++}$ for the first two distinct positively indexed subgroups in the filtration. In other words, we use $+$ as shorthand for $0+$, and $++$ as shorthand for $r_1+$ where $r_1 \in \RR$ is such that $G(k)_{x,0+}=G(k)_{x,r_1} \neq G(k)_{x,r_1+}$. When $r<s$ in $\tilde{\RR}$, we write $G(k)_{x,r:s}$ for the quotient $G(k)_{x,r}/G(k)_{x,s}$.

%We will frequently drop $k$ from the notation.

From Bruhat--Tits theory \cite[4.6.2, 5.2.6]{BT2} we may view $G(k)_{x,0}$ as the $\cO_k$-points of a smooth connected affine $\cO_k$-group scheme $\cG_x$ containing the split torus $T_{\cO_k}$. Let $\GG_x$ denote the special fibre of $\cG_x$, a split reductive group over $\FF_q$, so that $\GG_x(\FF_q) \cong G(k)_{x,0:0+}$. Let $\TT_x \leq \GG_x$ denote the special fibre of $T_{\cO_k}$, a split maximal torus of $\GG_x$ over $\FF_q$. For any $r > 0$, the quotient $G(k)_{x,r:r+}$ may be viewed as an $\FF_q$-vector space, with a natural choice of basis given by the image of $u_\psi(1) \in G(k)_{x,r}$ over all affine roots $\psi$ with $\psi(x)=r$.

One may replace $k$ with any unramified extension in the definitions above. Fix a Frobenius morphism $F:G(k^\ur) \to G(k^\ur)$ so that $G(k)=G(k^\ur)^F$. As we assume that $G$ splits over $k$, we have $G(k)_{x,r} = G(k^\ur)_{x,r}^F$ for all $r$. We also have the isomorphism
$$G(k^\ur)_{x,r:r+} \cong G(k)_{x,r:r+} \otimes_{\FF_q} \overline{\FF}_q$$ 
as representations of $\GG_x(\overline{\FF}_q)$, when $r>0$. We say that $G(k^\ur)_{x,r:r+}$ admits an algebraic representation of $\GG_x$ defined over $\FF_q$.

Denote by $\Delta(x)$ the set of affine roots taking minimal positive value at $x$, and denote this value by $\delta(x)$. Let $\Pi(x)=\Delta(x) \cap \Pi$. As we assume that $x$ is the barycentre of a facet, $\Pi(x)$ is the subset of $\Pi$ not vanishing at $x$. Let $W(x)$ be the subgroup of $W(G)$ stabilising $x$. When $x$ is the barycentre of $\cC$, we have $\Delta(x)=\Pi$. In general, $\Delta(x)$ contains $w(\Pi(x))$ for all $w \in W(x)$. We define the following subsets of $\Psi$: let $\Psi_x^+$ be the set of affine roots taking a positive value at $x$, let $\Psi_{x,<1}^+$ be the subset taking value at $x$ strictly between 0 and 1, and let $\Psi_{x,\leq 1}^+$ be the subset satisfying $0<\psi(x)\leq 1$. We have 
$$G(k)_{x,+:++} \cong \bigoplus\limits_{\psi \in \Delta(x)} V_\psi$$
as $\cO_k$-modules. As $x$ is the barycentre of a facet, every affine root evaluated at $x$ takes value in $\delta(x) \ZZ$. Let $\Delta_2(x)$ denote the set of affine roots taking value $2\delta(x)$ at $x$.

\subsection{The abelianisation of $G(k)_{x,+:1}$}

The constructions in \cite{GR10} and \cite{RY14} take as input characters of the quotients $G(k)_{x,+:++}$. Certainly any character of $G(k)_{x,+}$ factors through the abelianisation of $G(k)_{x,+}$. We have the following elementary calculation.

\begin{lemma}\label{lem:sl2}
    Suppose $G=\SL_2$ or $\mathrm{PGL}_2$ and let $x$ be the barycentre of an alcove so that $\cI:=G(k)_{x,0}$ 
    is an Iwahori subgroup. Denote the next two terms of the Moy--Prasad filtration by $\cI_+$ and $\cI_{++}$. Then the derived subgroup of $\cI_+$ is contained in $\cI_{++}$. Suppose furthermore that $p>2$. Then the abelianisation of $\cI_+$ is $\cI_+/\cI_{++}$.
\end{lemma}
\begin{proof}
    We prove this for $G=\SL_2$. A similar calculation holds for $\mathrm{PGL}_2$. After conjugating, we may assume that $\cI$ is the standard Iwahori subgroup generated by $T_0$, $U_{\alpha+0}$ and $U_{1-\alpha}$, where $U_{\alpha+0}$ denotes the strictly upper triangular matrices with entries in $\cO_k$, so that $\alpha$ is the standard choice of simple root of $\SL_2$. The following commutator relations 
    \begin{align*}
        \left[\begin{pmatrix}
            1+\varpi t & 0 \\ 0 & (1+\varpi t)^{-1}
        \end{pmatrix}, \begin{pmatrix}
            1 & b \\ 0 & 1
        \end{pmatrix}\right] &= \begin{pmatrix}
            1& \varpi(2tb+\varpi t^2b) \\ 0 & 1
        \end{pmatrix} \in \cI_{++} \\
        \left[\begin{pmatrix}
            1+\varpi t & 0 \\ 0 & (1+\varpi t)^{-1}
        \end{pmatrix}, \begin{pmatrix}
            1 & 0 \\ \varpi c & 1
        \end{pmatrix}\right] &= \begin{pmatrix}
            1& 0 \\ \varpi^2(2tc+\varpi t^2c) & 1
        \end{pmatrix} \in \cI_{++} \\
        \left[\begin{pmatrix}
            1 & b \\ 0 & 1
        \end{pmatrix}, \begin{pmatrix}
            1 & 0\\ \varpi c & 1
        \end{pmatrix}\right] &= \begin{pmatrix}
            1+\varpi bc + \varpi^2b^2c^2& -\varpi b^2c\\ \varpi^2 bc^2 & 1-\varpi bc
        \end{pmatrix} \in \cI_{++}
    \end{align*}
    for $b,c,t \in \cO_k$ imply that the derived subgroup $(\cI_+)_\mathrm{der}$ is contained in $\cI_{++}$. In other words, we recover the fact that $\cI_+/\cI_{++}$ is abelian. When $p>2$, the first two equations imply that $(\cI_+)_{\mathrm{der}}$ contains $U_{\alpha+1}$ and $U_{2-\alpha}$. Multiplying $\begin{psmallmatrix}
        1+\varpi bc + \varpi^2b^2c^2& -\varpi b^2c\\ \varpi^2 bc^2 & 1-\varpi bc
    \end{psmallmatrix}$ on the left and right by appropriate elements of $U_{2-\alpha}$ and $U_{\alpha+1}$ respectively, and varying $b,c$, we find that $(\cI_+)_\der$ contains $T_1$. It follows that $(\cI_+)_\mathrm{der} \geq \cI_{++}$, when $p>2$, and so the abelianisation of $\cI_+$ is $\cI_+/\cI_{++}$.
\end{proof}

\begin{cor}\label{cor:derived}
    When $p>2$, the derived subgroup of $G(k)_{x,+}$ contains $G(k)_{x,1+}$.
\end{cor}
\begin{proof}
    For each root $\alpha \in \Phi$, the algebraic group $\langle U_\alpha,U_{-\alpha} \rangle$ is isomorphic to $\SL_2$ or $\mathrm{PGL}_2$.

    If $\alpha(x) \not\in \ZZ$ then $$\langle U_\alpha,U_{-\alpha} \rangle \cap G(k)_{x,+} = \langle U_{\alpha-\lceil\alpha(x)\rceil+1}, U_{\lceil\alpha(x)\rceil-\alpha}, \alpha^\vee(1+\varpi\cO_k)\rangle.$$
    This is the pro-unipotent radical of an Iwahori subgroup of $\SL_2$ or $\mathrm{PGL}_2$. By Lemma \ref{lem:sl2}, the derived subgroup of this contains $\langle U_{\alpha-\lceil\alpha(x)\rceil+2}, U_{\lceil\alpha(x)\rceil-\alpha+1}, \alpha^\vee(1+\varpi\cO_k)\rangle.$

    If $\alpha(x) \in \ZZ$ then $$\langle U_\alpha,U_{-\alpha} \rangle \cap G(k)_{x,+} = \langle U_{\alpha-\alpha(x)+1}, U_{\alpha(x)-\alpha+1}, \alpha^\vee(1+\varpi\cO_k)\rangle.$$ This is the pro-unipotent radical of a maximal parahoric subgroup of $\SL_2$ or $\mathrm{PGL}_2$ (so is conjugate to the kernel of reduction modulo $\varpi$ of $\SL_2(\cO_k)$ or $\mathrm{PGL}_2(\cO_k)$). By the same calculation as in Lemma \ref{lem:sl2} (replacing $b$ by $\varpi b$), the derived subgroup of this contains $\langle U_{\alpha-\lceil\alpha(x)\rceil+2}, U_{\lceil\alpha(x)\rceil-\alpha+2}, \alpha^\vee(1+\varpi^2\cO_k)\rangle.$

    Combining all of this, we see that the derived subgroup of $G(k)_{x,+}$ contains $T_2$ and $U_{\psi+1}$ for all $\psi \in \Psi$ satisfying $\psi(x)>0$, and so contains $ G(k)_{x,1+}$.
\end{proof}
%"same calculation" since we can just rewrite everything in terms of the u_\alpha
% we just need all the coroots and not coweights since the filtration of T is determined by evaluating at all weights not roots.

For $p>2$, the abelianisation of $G(k)_{x,+}$ is therefore the same as the abelianisation of $G(k)_{x,+:1+}$. To simplify calculations, for all $p$ we will instead describe the abelianisation of $G(k)_{x,+:1}$. When $p>3$ (or $p>2$ and $G$ is not of type $G_2$) we will show that this abelianisation is simply $G(k)_{x,+:++}$. In small residue characteristic the abelianisation of $G(k)_{x,+:1}$ may be larger, and thus allows us to define more characters of $G(k)_{x,+}$. In this way, after compactly inducing, we obtain new supercuspidal representations of $G(k)$.

% to get the rest want to find T_1 intersect the derived subgroup... not necessarily all of T_1...

To determine the abelianisation of $G(k)_{x,+:1}$ we must describe the derived subgroup. Note that if $0<\psi(x),\phi(x)<1$ and $\psi+\phi$ is constant (necessarily 1), then $[U_\psi,U_\phi] \subset \langle U_{\psi+1},U_{\phi+1}, T_1 \rangle \leq G(k)_{x,1}$ by the calculation of Lemma \ref{lem:sl2}. Chevalley's commutator formula for affine roots \cite[Lemma 15]{Ste16} gives, for affine roots $\psi,\phi$ with $\psi+\phi$ non-constant and $0<\psi(x),\phi(x)<1$, 
$$[u_{\psi}(r),u_{\phi}(s)] = \prod\limits_{i,j>0 \hspace{1mm} : \hspace{1mm} i\psi+j\phi \in \Psi_{x,< 1}^+} u_{i\psi+j\phi}(C_{\dot\psi,\dot\phi,i,j}(-r)^i s^j) \mod G(k)_{x,1}$$
taken with respect to increasing order in $i+j$ (the terms with $i+j=3$ commute), for all $r,s \in \cO_k$, and for explicit constants $C_{\dot\psi,\dot\phi,i,j} \in \{\pm 1, \pm 2, \pm 3\}$. 

%If $G$ is not of type $G_2$ then $C_{\dot\psi,\dot\phi,i,j} \in \{\pm 1, \pm 2\}$. If $G$ is not of type $B_n, C_n, F_4, G_2$ then $C_{\dot\psi,\dot\phi,i,j} \in \{\pm 1\}$.

We state some properties of these constants:

\begin{lemma}\label{lem:constants}
Let $\alpha, \beta$ be linearly independent roots of $G$. Then 
\begin{align*}
    C_{\alpha \beta i 1} &= M_{\alpha \beta i} \\
    C_{\alpha \beta 1 j} &= (-1)^j M_{\beta \alpha j} \\
    C_{\alpha \beta 3 2} &= \frac{1}{3}M_{\alpha+\beta,\alpha,2} \\
    C_{\alpha \beta 2 3} &= -\frac{2}{3}M_{\alpha+\beta,\beta, 2}
\end{align*}
where if the $\alpha$-string through $\beta$ is $\beta-p\alpha, \dots, \beta + q\alpha$, then $$M_{\alpha \beta i} = \pm \binom{p+i}{i}$$ where the sign depends on choices made in defining a Chevalley basis for $G$. In particular $C_{\alpha,\beta,i,j} \in \{\pm 1, \pm 2, \pm 3\}$ and if $\beta-\alpha$ is not a root then $C_{\alpha,\beta,i,1}$ and $C_{\alpha,\beta,1,j}$ are $\pm 1$.
\end{lemma}
\begin{proof}
    In general, the constants in the commutator formula are unique and depend only on the roots of a simple reductive group $H$ (more precisely, their structure constants) \cite[Lemma 15]{Ste16}, and not the isogeny type of $H$. In particular, the Chevalley constants for $G$ are the same as for the adjoint group $G_{\mathrm{ad}}$, and these are given in \cite[Section 4.3 and Theorem 5.2.2]{Car1}.
\end{proof}

%This is because, for a general group, the commutator subgroup is generated by the commutators of a set of generators. For more detail see \cite[Theorem 4]{Gas20}. 

\begin{lemma}\label{lem:delta_span}
    The gradients of $\bigcup\limits_{w \in W(x)} w(\Pi(x))$ span $X^*(T) \otimes \RR$.
\end{lemma}
\begin{proof}
    Let $(,)$ denote the bilinear form on $X^*(T) \otimes \RR$ determined by the Killing form on $\mathrm{Lie}(T)$. Let $\dot\Pi$ and $\dot\Pi(x)$ denote the set of gradients of $\Pi$ and $\Pi(x)$ respectively. As $\dot\Pi$ contains the simple roots of $\Phi$, and thus spans $X^*(T) \otimes \RR$, it suffices to show that the $\RR$-span of the gradients of $\bigcup\limits_{w \in W(x)} w(\Pi(x))$ contains $\dot\Pi$. Certainly it contains $\dot\Pi(x)$. Note that $\Pi(x)$ is always non-empty. Let $\psi \in \Pi\backslash\Pi(x)$. As the affine Dynkin diagram of $G$ is path-connected (as $G$ is simple), there exists $\psi=\psi_1,\psi_2,\dots,\psi_n \in \Pi \backslash \Pi(x)$ and $\psi_{n+1} \in \Pi(x)$ such that $(\dot\psi_i,\dot\psi_j)$ is negative if $|i-j|=1$ and is zero if $|i-j|>1$. 
    
    Then, for all $1 \leq r \leq n$, we have $$w_{\psi_r} w_{\psi_{r+1}} \dots w_{\psi_n}(\psi_{n+1}) \in \bigcup\limits_{w \in W(x)} w(\Pi(x)),$$ where $w_{\psi_r}$ denotes the simple reflection in $W(G)$ corresponding to $\psi_r$. Note that $w_{\psi_r} \in W(x)$ as $\psi_r \in \Pi\backslash \Pi(x)$. The difference 
    $$w_{\psi_r} w_{\psi_{r+1}} \dots w_{\psi_n}(\psi_{n+1}) - w_{\psi_{r+1}} w_{\psi_{r+2}} \dots w_{\psi_n}(\psi_{n+1})$$ is of the form $a_r\psi_r+b_r$ for some integers $a_r,b_r$, with $a_r>0$. Taking gradients, it follows that $\dot\psi_r$ is in the $\RR$-span of the gradients of $\bigcup\limits_{w \in W(x)} w(\Pi(x))$ for all $1 \leq r \leq n+1$. In particular, $\dot\psi$ is in the $\RR$-span of the gradients of $\bigcup\limits_{w \in W(x)} w(\Pi(x))$, completing the proof.
\end{proof}

\begin{cor}
    The gradients of $\Delta(x)$ span $X^*(T) \otimes \RR$.
\end{cor}

\begin{lemma}
    There exist positive integers $a_\theta$ such that $$\sum\limits_{\theta \in \bigcup\limits_{w \in W(x)}w(\Pi(x))} a_\theta \theta$$ is a constant function.
\end{lemma}
\begin{proof}
    There exist unique positive integers $b_\psi$ such that $\sum_{\psi \in \Pi}b_\psi \psi = 1$. View the left hand side of this equation as an element of the free abelian group $\ZZ[\Pi]$ on the set $\Pi$. This naturally embeds into the free abelian group $\ZZ[\bigcup\limits_{w \in W(x)} w(\Pi)]$. Define 
    $$\sigma: = \sum\limits_{w \in W(x)} \sum\limits_{\psi \in \Pi} b_\psi w(\psi) \in \ZZ[\bigcup\limits_{w \in W(x)} w(\Pi)].$$ Then all coefficients of $\sigma$ are positive integers. Denote the coefficient of $\theta$ in $\sigma$ by $\sigma(\theta)$. The elements of $\bigcup\limits_{w \in W(x)} w(\Pi \backslash \Pi(x))$ are exactly the affine roots vanishing at $x$. For such an affine root $\theta$, let $w_\theta$ denote the corresponding reflection in $W(x)$. Then, as $\sigma$ is invariant under the action of $w_\theta$, and $w_\theta(\theta)=-\theta$, we observe that $\sigma(\theta)=\sigma(-\theta)$. By construction, the sum of affine roots, 
    $$\sum\limits_{\theta \in \bigcup\limits_{w \in W(x)} w(\Pi)} \sigma(\theta) \theta,$$ is a constant function. As $\sigma(\theta)=\sigma(-\theta)$ for $\theta \in \bigcup\limits_{w \in W(x)} w(\Pi \backslash \Pi(x))$, it follows that the sum of affine roots, 
    $$\sum\limits_{\theta \in \bigcup\limits_{w \in W(x)} w(\Pi(x))} \sigma(\theta) \theta,$$ is a constant function. We have seen that all such $\sigma(\theta)$ are positive integers, and we take $a_\theta=\sigma(\theta)$.
\end{proof}

\begin{cor}\label{cor:delta_span}
    For any positive affine root $\psi \in \Psi_{x,\leq 1}^+ \backslash \Delta(x)$ there exists $\theta \in \Delta(x)$ such that $\psi-\theta \in \Psi_{x,\leq 1}^+$.
\end{cor}
\begin{proof}
    The above lemma gives the equality 
    $$\sum\limits_{\theta \in \bigcup\limits_{w \in W(x)} w(\Pi(x))}a_\theta (\dot\psi,\dot\theta) = 0.$$
    As the Killing form is positive definite, by Lemma \ref{lem:delta_span} we must have that $\dot\psi$ pairs positively with $\dot\theta$ for some $\theta \in \bigcup\limits_{w \in W(x)} w(\Pi(x)) \subset \Delta(x)$, with respect to $(,)$. Since $0<\theta(x) < \psi(x)\leq 1$, we know that $\dot\psi \neq \dot\theta$. So $\psi-\theta \in \Psi$, and $0< (\psi-\theta)(x) \leq 1$. 
\end{proof}

\begin{notn}
    For $\psi \in \Psi_{x}^+$, let $\overline{U}_\psi$ denote the image of $U_\psi$ in $G(k)_{x,+:1}$.
\end{notn}

\begin{prop}
    The following hold:
    \begin{itemize}
        \item If $G$ is simply laced then the abelianisation of $G(k)_{x,+:1}$ is $G(k)_{x,+:++}$.
        \item If $G$ is of type $B_n,C_n,F_4$ and $p>2$ then the abelianisation of $G(k)_{x,+:1}$ is $G(k)_{x,+:++}$.
        \item If $G$ is of type $G_2$ and $p>3$ then the abelianisation of $G(k)_{x,+:1}$ is $G(k)_{x,+:++}$.
    \end{itemize}
\end{prop}
    \begin{proof}
        Let $\psi \in \Psi_x^+$. Then $\psi(x) \in \delta(x)\ZZ$. We must show that if $\psi(x) > \delta(x)$ then $\overline{U}_\psi \subset \cpder$. By definition, when $\psi(x)\geq 1$, $\overline U_\psi$ is trivial. In particular, $\overline U_\psi$ is contained in $\cpder$ when $\psi(x) \geq \delta(x)m$ for any integer $m \geq 1/\delta(x)$. It then suffices to show that for any integer $m \geq 2$, if $\overline{U}_\psi \subset \cpder$ whenever $\psi(x)>\delta(x)m$ then $\overline{U}_\psi \subset \cpder$ when $\psi(x)\geq \delta(x)m$. 
        
        Suppose $1 \geq \psi(x) \geq \delta(x)m > \delta(x)$. Then $\psi \not\in \Delta(x)$. By Corollary \ref{cor:delta_span} there exists $\theta \in \Delta(x)$ such that $\psi-\theta \in \Psi_x^+$. Applying the commutator relations to $\theta$ and $\psi-\theta$, we see that 

        $$[u_{\theta}(r),u_{\psi-\theta}(s)] = \prod\limits_{i,j>0 \hspace{1mm} : \hspace{1mm} i\theta+j(\psi-\theta) \in \Psi_{x,< 1}^+} u_{i\theta+j(\psi-\theta)}(C_{\dot\theta,\dot\psi-\dot\theta,i,j}(-r)^i s^j) \mod G(k)_{x,1}$$
        for all $r,s \in \cO_k$. When $i>1$ or $j>1$, $(i\theta+j(\psi-\theta))(x) >\psi(x)$, and so by the induction hypothesis $\overline U_{i\theta+j(\psi-\theta)} \leq (G(k)_{x,+:1})_\der$. As $[u_{\theta}(r),u_{\psi-\theta}(s)] \mod G(k)_{x,1}\in \cpder$, it follows that 
        $$u_{\psi}(C_{\dot\theta,\dot\psi-\dot\theta,1,1}(-r) s) \mod G(k)_{x,1} \in \cpder$$
        for all $r,s \in \cO_k$. By assumption, $C_{\dot\theta,\dot\psi-\dot\theta,1,1}$ is invertible in $\cO_k$ and we deduce that $\overline U_\psi\leq \cpder$. Hence $(G(k)_{x,+:1})_\der \geq G(k)_{x,++:1}$, while certainly $G(k)_{x,+:++}$ is abelian.
    \end{proof}

    \iffalse
    \begin{rem}
        When $p>2$ in the above proposition, we may replace the abelianisation of $G(k)_{x,+:1}$ with the abelianisation of $G(k)_{x,+}$ in each statement. This almost follows from combining this proposition with Corollary \ref{cor:derived}. This stronger claim holds by replacing $G(k)_{x,1}$ in the above proof everywhere by the the more cumbersome group $\langle T_1,G(k)_{x,1+}\rangle$.
    \end{rem}

\fi

We now consider the remaining cases when $p=2,3$.

\begin{notn}
    Let $\psi,\theta \in \Psi_x^+$ such that $\overline U_\psi$ and $\overline U_\theta$ commute in $G(k)_{x,+:1}$. Let $\overline U_{\psi,\theta}$ denote the subgroup of $G(k)_{x,+:1}$  $$\overline U_{\psi,\theta} = \{u_\psi(r)u_\theta(r) \mod G(k)_{x,1} : r\in \cO_k\}.$$
\end{notn}

\begin{lemma}\label{lem:conditions}
    Suppose $G$ is not of type $G_2$, $p=2$, and let $\psi \in \Psi_x^+ \backslash \Delta(x)$.
    \begin{enumerate}
        \item If there exists $\theta \in \Delta(x)$ such that $$\pair{\dot\psi}{\dot\theta^\vee} = 1 = \pair{\dot\theta}{\dot\psi^\vee},$$ then $\overline U_\psi \leq \cpder$.
        \item Suppose there exists $\theta \in \Delta(x)$ such that $\pair{\dot\psi}{\dot\theta^\vee} = 1$ and $ \pair{\dot\theta}{\dot\psi^\vee}=2$. Then $2\psi-\theta \in \Psi_x^+$, and $\overline U_\psi$ and $\overline U_{2\psi-\theta}$ commute in $G(k)_{x,+:1}$. Moreover, if $q \neq 2$ then $\overline U_\psi, \overline U_{2\psi-\theta} \leq \cpder$. If $q=2$, then $\overline U_{\psi,\theta} \leq \cpder$.
    \end{enumerate}
\end{lemma}
\begin{proof}
    Since $G$ is not of type $G_2$, any root string in $\mathrm{Lie}(G)$ has length at most 3.
    \begin{enumerate}
        \item The conditions on $\theta$ imply that the root string of $\dot\theta$ through $\dot\psi$ is $\dot\psi-\dot\theta, \dot\psi$, and the root string of $\dot\psi$ through $\dot\theta$ is $\dot\theta-\dot\psi, \dot\theta$. Then the commutator formula for $\theta$ and $\psi-\theta$ gives 
        $$[u_{\theta}(r),u_{\psi-\theta}(s)] = u_\psi(C_{\dot\theta,\dot\psi-\dot\theta,1,1}(-r)s) \mod G(k)_{x,1}$$
        for all $r,s \in \cO_k$. By Lemma \ref{lem:constants}, $C_{\dot\theta,\dot\psi-\dot\theta,1,1}=\pm 1$, and we deduce that $\overline U_\psi \leq \cpder$.
        \item As $\psi \in \Psi_x^+\backslash \Delta(x)$, we have $\psi(x) \geq 2\delta(x)$, so $(2\psi-\theta)(x) \geq \delta(x)>0$ and therefore $2\psi-\theta \in \Psi_x^+$. The subgroups $\overline U_\psi$ and $\overline U_{2\psi-\theta}$ of $G(k)_{x,+:1}$ commute as an immediate consequence of the commutator formula, as $3\dot\psi-\dot\theta \not\in \Phi$ (as root strings have length at most 3). The commutator formula also gives
        $$[u_{\theta}(r),u_{\psi-\theta}(s)] =  u_\psi((-r)s) \cdot  u_{2\psi-\theta} ((-r)s^2) \mod G(k)_{x,1}$$ for all $r,s \in \cO_k$. Thus $$u_\psi((-r)s) \cdot  u_{2\psi-\theta} ((-r)s^2) \mod G(k)_{x,1} \in \cpder$$ for all $r,s \in \cO_k$. When $q=2$, $s^2=s \mod \varpi$ for all $s \in \cO_k$, and so $\overline U_{\psi,\theta} \leq \cpder$. When $q>2$, let $e \in \FF_q^\times$ with $e^2 \neq e$. Fix $a,b \in \FF_q$. Consider pairs  
        $$(r,s) \equiv\left(\frac{a-b}{e^2-e},e\right) \mod \varpi \text{ and } (r',s') \equiv \left(\frac{ae-b}{1-e},1\right) \mod \varpi.$$
        Then 
        $$[u_\theta(r),u_{\psi-\theta}(s)][u_\theta(r'),u_{\psi-\theta}(s')] = u_\psi(\tilde{a})u_{2\psi-\theta}(\tilde{b}) \mod G(k)_{x,1} \in \cpder$$
        where $\tilde{a},\tilde{b}$ are lifts of $a,b$ to $\cO_k$. Considering the cases when $a=0$ or $b=0$, it follows that $\overline U_\psi, \overline U_{2\psi-\theta} \leq \cpder$.

    \end{enumerate}
\end{proof}

\begin{lemma}\label{lem:notG2}
    Suppose $p=2$ and that $G$ is not of type $G_2$. Let $S(x) \subset \Psi_x^+$ be the set 
    \begin{align*}
    S(x) := \Delta(x) &\sqcup \{\psi+\theta: \psi,\theta \in \Delta(x), \psi+\theta \in \Psi\} \\ 
    &\sqcup \{2\psi+\theta: \psi,\theta \in \Delta(x), 2\psi+\theta \in \Psi\} 
\end{align*}
    For any $\psi \in \Psi_x^+ \backslash S(x)$, we have $\overline U_\psi \leq \cpder$.
\end{lemma}
\begin{proof}
    Let $\psi \in \Psi_x^+ \backslash S(x)$. Then $\psi(x) \in\delta(x)\ZZ$. Moreover, as $\psi \not\in \Delta(x)$, we have $\psi(x) \geq 2\delta(x)$. If $\psi(x) \geq 1$, then by definition $\overline U_\psi$ is trivial and so contained in $\cpder$. To prove the claim by induction on integers $m \geq 2$, it suffices to prove that, assuming $\overline U_\psi \leq \cpder$ whenever $\psi(x)>m\delta(x)$, we also have $\overline U_\psi \leq \cpder$ whenever $\psi(x)\geq m\delta(x)$. Suppose $\psi(x)= m\delta(x)$. We divide into cases depending on whether $\dot\psi$ is a short or long root.

    Suppose $\dot\psi$ is short. Then by Corollary \ref{cor:delta_span}, there exists $\theta \in \Delta(x)$ such that $\pair{\dot\psi}{\dot{\theta}^\vee} >0$, which is necessarily 1 as $\dot\psi$ is short. If $\pair{\dot\theta}{\dot\psi^\vee} = 1$, then $\overline U_\psi \leq \cpder$ by Lemma \ref{lem:conditions}. Otherwise, $2\dot\psi-\dot\theta \in \Phi$. Since $\psi \not\in \Delta(x)$, we have $(2\psi-\theta)(x)>\psi(x) \geq m\delta(x)$. If $2\psi-\theta \not\in S(x)$, then the induction hypothesis implies that $\overline U_{2\psi-\theta} \leq \cpder$, and this forces $\overline U_\psi \leq \cpder$. Otherwise, $2\psi-\theta$ is of the form $\psi_1+\psi_2$ or $2\psi_1+\psi_2$ for $\psi_1,\psi_2 \in \Delta(x)$. In the first case, $\psi-\theta \in \Psi$ takes value $\delta(x)/2$ at $x$, contradicting the minimality of $\delta(x)$. In the second case, the equality $(\psi-\theta)(x)=\delta(x)$ means that $\psi-\theta \in \Delta(x)$ by definition. But $\psi=\theta+(\psi-\theta)$  contradicts our assumption that $\psi \not\in S(x)$. 

    Suppose now $\dot\psi$ is long. Then by Corollary \ref{cor:delta_span}, there exists $\theta \in \Delta$ such that $\pair{\dot\psi}{\dot{\theta}^\vee} \in \{1,2\}$. If the value is 1, the same argument as above holds to show that $\overline U_\psi \leq \cpder$. If the value is 2 then the commutator formula gives 
    \begin{equation}\label{eqn:2}
        [u_{\psi-2\theta}(r),u_{\theta}(s)] =  u_{\psi-\theta}((-r)s) \cdot  u_{\psi} ((-r)s^2) \mod G(k)_{x,1}
    \end{equation}
    for all $r,s \in \cO_k$. Consider first the case that $\psi-\theta \in S(x)$. If $\psi-\theta \in \Delta(x)$ then $\psi =\theta+ (\psi-\theta) \in S(x)$. If $\psi-\theta = \psi_1+\psi_2$ for some $\psi_1,\psi_2 \in \Delta(x)$, then $\psi-2\theta \in \Psi$ takes value $\delta(x)$ at $x$, so is in $\Delta(x)$, and so $\psi = 2\theta + (\psi-2\theta) \in S(x)$. The case that $\psi-\theta = 2\psi_1+\psi_2$ for some $\psi_1,\psi_2 \in \Delta(x)$ does not occur as the gradient of $\psi-\theta$ is a short root, and the gradient of $2\psi_1+\psi_2$ is a long root. In any case, we have a contradiction to $\psi \not\in S(x)$. Finally, consider the case that $\psi-\theta \not\in S(x)$. The gradient $\dot\psi-\dot\theta$ is a short root, and again by Corollary \ref{cor:delta_span} there exists $\phi \in \Delta(x)$ such that $\pair{\dot\psi-\dot\theta}{\dot\phi^\vee}=1$. If $\dot\phi$ is a short root, then Lemma \ref{lem:conditions} implies that $\overline U_{\psi-\theta} \leq \cpder$. Then $\overline U_\psi \leq \cpder$ by (\ref{eqn:2}). Otherwise, $2\dot\psi-2\dot\theta-\dot\phi \in \Phi$. We have $(2\psi-2\theta-\phi)(x) = (2m-3)\delta(x)$. If $m>3$ then $\overline U_{2\psi-2\theta-\phi} \leq \cpder$ by the induction hypothesis, and then Lemma \ref{lem:conditions} forces $\overline U_{\psi-\theta} \leq \cpder$, and again we have $\overline U_\psi \leq \cpder$. The case $m=2$ implies $\psi-\theta \in \Delta(x)$, contradicting $\psi-\theta \not\in S(x)$. The case $m=3$ implies $\psi-2\theta \in \Delta(x)$, and then $\psi = 2\theta+(\psi-2\theta)$ contradicts $\psi \not\in S(x)$.
\end{proof}

We will show that the abelianisation $G(k)_{x,+:1}^\ab$ admits the structure of an $\FF_q$-vector space. Moreover, a basis for this vector space is naturally indexed by a finite set $\cS(x)$ that we will now define algorithmically. The set $\cS(x)$ consists of elements that are either affine roots $\psi$ or tuples of affine roots $(\psi_1,\psi_2,\dots,\psi_n)$. We continue to assume that $G$ is not of type $G_2$ and $p=2$.

We initialise $\cS(x)$ by taking 
\begin{align*}
    \cS(x)' := \Delta(x) &\sqcup \{\psi+\theta: \psi,\theta \in \Delta(x), \psi+\theta \in \Psi\} \\ 
    &\sqcup \{2\psi+\theta: \psi,\theta \in \Delta(x), 2\psi+\theta \in \Psi\} 
\end{align*}
to be the set $S(x)$ from the previous lemma.

For all pairs $(\psi,\theta) \in \Delta(x)^2$ we apply the following. Suppose $\dot\psi+\dot\theta \in \Phi$ and that $\dot\psi-\dot\theta \not\in \Phi$.

\begin{itemize}
    \item If $\pair{\dot\psi}{\dot\theta^\vee} = \pair{\dot\theta}{\dot\psi^\vee} = -1$, then remove $\psi+\theta$ from $\cS(x)'$.
    \item If $\pair{\dot\psi}{\dot\theta^\vee} = -1$ and $\pair{\dot\theta}{\dot\psi^\vee} = -2$, then add $(\psi+\theta,2\psi+\theta)$ to $\cS(x)'$.
\end{itemize}

Now, for $\epsilon \in \Delta(x)$ and $\eta \in \Delta_2(x)$, if $\epsilon+\eta \in \cS(x)'$ and $\eta-\epsilon \not\in \Delta(x)$, then remove $\epsilon+\eta$ from $\cS(x)'$.

If $(\psi+\theta,2\psi+\theta) \in \cS(x)'$ and $\psi+\theta \not \in \cS(x)'$ or $2\psi+\theta \not\in \cS(x)'$, remove all of $\psi+\theta,2\psi+\theta,(\psi+\theta,2\psi+\theta)$ from $\cS(x)'$. Repeat this step over each pair $(\psi+\theta,2\psi+\theta) \in \cS(x)'$ until $\cS(x)'$ stabilises.

Finally, define a graph on the set of (singleton) affine roots in $\cS(x)'$ by adding an edge between $\psi$ and $\theta$ if $(\psi,\theta) \in \cS(x)'$. If a connected component of this graph consists of affine roots $\psi_1,\dots,\psi_n$, with $n>1$, add the (unordered) tuple $(\psi_1,\dots,\psi_n)$ to $\cS(x)'$ and remove all edges $(\psi_i,\psi_j)$ and vertices $\psi_i$ from $\cS(x)'$ appearing in this connected component. After doing this for each connected component, when $q=2$ we take $\cS(x)$ to be the resulting set $\cS(x)'$. When $q>2$ we take $\cS(x)$ to be the subset of (singleton) affine roots in $\cS(x)'$.

\begin{defn}
    We define the $\FF_q$-vector space $\cV(x)$ as the direct sum of lines $V_s$ over $s \in \cS(x)$, where when $s=\psi$ is an affine root this notation is consistent with the notation $V_\psi$, and when $s$ is a tuple $(\psi_1,\dots,\psi_n)$ we define $V_s$ by 
    $$V_s = \left(\bigoplus\limits_{i=1}^n V_{\psi_i}\right) \bigg{/} \langle \overline{u}_{\psi_i}(r)+\overline{u}_{\psi_j}(r) : i \neq j \text{ and } r \in \cO_k/\varpi \cO_k\rangle.$$
\end{defn}

\begin{notn}
    For $\psi \in \Psi_x^+$, write $\nu_\psi: U_\psi \to \cV(x)$ for the homomorphism that is identically 0 if $\psi \not\in \cS(x)$, is the quotient map $U_\psi \to V_\psi$ if $\psi \in \cS(x)$, and is the composition $$U_\psi \twoheadrightarrow V_\psi \to V_s$$ when $\psi$ belongs to a tuple $s \in \cS(x)$.
\end{notn}

\begin{prop}
    Suppose $p=2$ and $G$ is not of type $G_2$. There is a unique homomorphism $\nu: G(k)_{x,+} \to \cV(x)$ that is trivial on $T_1$ and is given by $\nu_\psi$ on each $U_\psi$. Moreover, $\nu$ is surjective and realises $\cV(x)$ as the abelianisation $V(x)$ of $G(k)_{x,+:1}$.
\end{prop}
\begin{proof}
    To show that $\nu$ is well-defined, by \cite[Theorem 4]{Gas20} (we may replace $\CC^\times$ with any abelian group, in particular $\cV(x)$), it suffices to show that, for all $\phi,\theta \in \Psi_{x,< 1}^+$, we have 
    \begin{equation}\label{eqn:commutator}
        0 = \sum\limits_{i,j>0: i\phi+j\theta \in \Psi_{x,< 1}^+} \nu_{i\phi+j\theta}(C_{\dot\phi,\dot\theta,i,j}(-r)^i s^j)
    \end{equation}
    in $\cV(x)$, for all $r,s \in \cO_k$. As $\nu_\psi$ is trivial when $\psi$ does not appear in $\cS(x)$ (as a singleton or a tuple), it suffices to check that (\ref{eqn:commutator}) holds whenever $\phi,\theta \in \Psi_{x,< 1}^+$ are such that there exists $i,j>0$ such that $i\phi+j\theta$ appears in $\cS(x)$. For any $\psi$ appearing in $\cS(x)$ we have $\psi(x) \leq 3\delta(x)$. So we must have $\phi,\theta \in \Delta(x),\Delta_2(x)$. Our tedious construction of $\cS(x)$ was designed exactly so that (\ref{eqn:commutator}) holds for all $\phi,\theta \in \Delta(x),\Delta_2(x)$. As $\nu_\psi$ is surjective when $\psi$ appears in $\cS(x)$, the homomorphism $\nu$ is surjective. Any homomorphism $\nu'$ from $G(k)_{x,+}$ to an abelian group satisfies the analogue of (\ref{eqn:commutator}), so that by construction $\nu'$ must factor through $\nu$. This is the universal property defining $\cV(x)$ as the abelianisation of $G(k)_{x,+:1}$.
\end{proof}

\begin{example}
    Suppose $x$ is the barycentre of $\cC$. We describe $\cS(x)$ in this setting. Initialise $\cS(x)'$ as described above. The set $\Delta(x)$ is simply the set $\Pi$ of simple affine roots. Any $\psi \in \Psi_x^+=\Psi^+$ may be expressed as a unique non-negative integer linear combination in $\Pi$. Suppose $\psi \in \cS(x)'$ is of the form $\psi_1+\psi_2$ for $\psi_1,\psi_2\in \Pi$. If $2\psi_1+\psi_2 \not \in \Psi$ and $\psi_1+2\psi_2 \not\in \Psi$, corresponding to the equalities $\pair{\dot\psi_1}{\dot\psi_2^\vee}= \pair{\dot\psi_2}{\dot\psi_1^\vee}=-1$, the commutator relation for $\psi_1$ and $\psi_2$ implies that $\overline{U}_\psi \leq \cpder$, which is why we remove $\psi=\psi_1+\psi_2$ from $\cS(x)'$ in the algorithm.

    Suppose $\psi \in \cS(x)'$ is of the form $2\psi_1+\psi_2$ for $\psi_1,\psi_2 \in \Psi$. Then we also have $\psi_1+\psi_2 \in \cS(x)'$. When $q>2$, the commutator relation for $\psi_1$ and $\psi_2$ implies that $\overline{U}_{\psi_1+\psi_2},\overline{U}_{2\psi_1+\psi_2} \leq \cpder$. When $q=2$, we instead have $\overline{U}_{\psi_1+\psi_2,2\psi_1+\psi_2}\leq \cpder$. The only simple affine roots $\psi_i,\psi_j$ such that $\psi_1+\psi_2$ or $2\psi_1+\psi_2$ appears as an integer linear combination of $\psi_i$ and $\psi_j$ are $\psi_1,\psi_2$ themselves. This explains why, in our algorithm, we remove $\psi_1+\psi_2$ and $2\psi_1+\psi_2$ from $\cS(x)'$ and, when $q=2$ only, we add back the pair $(\psi_1+\psi_2,2\psi_1+\psi_2)$. 

    After applying the above for all $\psi \in \cS(x)'$, the result is $\cS(x)$. For general points $x$, linear dependencies amongst $\Delta(x)$ lead to the additional steps described in the algorithm.
\end{example}

For an example when $G$ is of type $F_4$ and $x$ is not the barycentre of $\cC$, and for more details in the setting where $G$ is of type $G_2$, please see the author's forthcoming thesis. As $G_2$ has small rank, it is not difficult to compute the abelianisation $V(x)$ of $G(k)_{x,+:1}$ by hand for the few possibilities for $x$. This follows the same principles as the algorithm above. We still find that $V(x)$ admits the structure of an $\FF_q$-vector space. Some explicit examples are given in Section 4.

\section{Compactly induced representations and supercuspidals}\label{sec:intertwine}
We continue to let $G$ be a simple split reductive group over $k$ and $x$ be the barycentre of a facet in the closure of $\cC$. In the previous section we described the abelianisation $V(x)$ of $G(k)_{x,+:1}$. 

For the remainder of this section, fix a nontrivial additive character $\chi: \FF_q \to \CC^\times$. For any $\lambda \in V(x)^\vee$, we write $\chi_\lambda: G(k)_{x,+}\to \CC^\times$ for the inflation of $\chi \circ \lambda : V(x) \to \CC^\times$. We consider the compactly induced representation
$$\pi_\lambda := \ind_{G(k)_{x,+}}^{G(k)} \chi_\lambda$$
and describe a condition on $\lambda$ for which $\pi_\lambda$ decomposes as a finite direct sum of irreducible supercuspidal representations.

\subsection{Generalities on intertwining}

\begin{defn}
    Let $K$ be an open subgroup of $G(k)$ containing $Z(k)$. Let $(\tau,W)$ be a smooth finite-dimensional complex representation of $K$. For $g \in G(k)$, write $K^g=g^{-1}Kg$ and let $\tau^g:K^g \to \GL(W)$ be the representation given by $\tau^g(k^g)=\tau(k)$. The intertwining of $\tau$ is the set $$I(G,K,\tau) = \{g \in G(k): \Hom_{K \cap K^g}(\tau,\tau^g) \neq \{0\}\}.$$ This is a union of $K$-double cosets.
\end{defn}

% I think in Reeder-Yu don't need J to contain Z, see where Lemma 2.1 i s used in 2.2. Note: finite index is not a problem since Z is finite, and this assumption is made in that G is semisimple, e.g. GL_2 is not semisimple.

\begin{lemma}\label{lem:mackey_basic}
    Let $K$ be an open subgroup of $G(k)$ containing $Z(k)$. Let $(\tau,W)$ be a smooth finite-dimensional irreducible complex representation of $K$. Then $\ind_K^{G(k)} \tau$ is an irreducible representation of $G(k)$ if and only if $I(G,K,\tau)=K$.
\end{lemma}
\begin{proof}
    This is a generalisation of \cite[Theorem 11.4]{BH1}. See also \cite{Kut77}.
\end{proof}

More generally, suppose that $J$ is a compact open subgroup of $G(k)$ that is normal in some open subgroup $H$ of $G(k)$ that contains $Z(k)$ and is compact modulo $Z(k)$. For a smooth character $\chi: J \to \CC^\times$, define the subgroup $H_\chi=\{h \in H: \chi^h=\chi\}$ of $H$. Then $J$ is normal in $H_\chi$, and the quotient $A_\chi = H_\chi/J$ is a finite group when $Z(k)$ is compact, as is the case when $G$ is simple. Define $$\cH_\chi := \End_{H_\chi}(\ind_J^{H_\chi} \chi).$$

By Mackey theory for finite groups, there is a bijection $\rho \mapsto \chi_{\rho}$ between the set $\mathrm{Irr}(\cH_\chi)$ of simple $\cH_\chi$-modules to the set of irreducible constituents of $\ind_{J}^{H_\chi} \chi$ such that 
$$\ind_{J}^{H_\chi} \chi = \bigoplus\limits_{\rho \in \mathrm{Irr}(\cH_\chi)} (\dim \rho) \cdot \chi_{\rho}.$$
The $\chi_{\rho}$ are $\chi$-isotypic when restricted to $J$ and induce irreducibly to $H$.

% The Mackey theory is telling us the $\chi$-isotypic part, and from this also the fact the the number of copies is dim \rho. The part about simple modules is just Wedderburn theory - given a decomposition of the induction we can write down the structure of the endomorphism algebra, and the simple modules (acting on itself is the permutation rep, and any simple modules appears in the permutation rep)

\begin{lemma}\label{lem:mackey}
    Assume that $I(G,J,\chi)=H_\chi$. Then the following hold:
    \begin{itemize}
        \item The representation $\ind_J^{G(k)} \chi$ has a finite direct sum decomposition $$\ind_J^{G(k)} \chi = \bigoplus\limits_{\rho \in \mathrm{Irr}(\cH_\chi)} (\dim \rho) \cdot \ind_{H_\chi}^{G(k)} \chi_\rho.$$ 
        %where the $\chi_\rho$ are the distinct irreducible constituents of $\ind_{J}^{H_\chi} \chi$.
        \item For each $\rho \in \mathrm{Irr}(\cH_\chi)$, the compactly induced representation $\ind_{H_\chi}^{G(k)} \chi_\rho$ is irreducible and supercuspidal.
        \item If $\rho,\rho'$ are inequivalent simple modules for $\cH_\chi$, then $\ind_{H_\chi}^{G(k)}\chi_\rho$ and $\ind_{H_\chi}^{G(k)}\chi_{\rho'}$ are inequivalent representations of $G(k)$.
        
    \end{itemize}
\end{lemma}
\begin{proof}
    See, for example, \cite[Lemma 2.2]{RY14}.
\end{proof}

The following structural result will be useful in determining intertwining sets.

\begin{prop}\label{prop:bruhat}
    Let $x,y$ be points in $\cA$. We have a version of the affine Bruhat decomposition, 
    $$G(k)=G(k)_{x,0}N(k)G(k)_{y,0}.$$
    A set of double coset representatives is contained in a set of lifts of the Iwahori--Weyl group $W(G)$ to $N(k)$.
\end{prop}
\begin{proof}
    For the case when $x=y$ is the barycentre of an alcove, see \cite[3.3.1]{Tit79}. See also \cite{IM1}. This case also implies the decomposition holds when $x$ and $y$ lie in the closure of the same alcove, as if the barycentre of such an alcove is $z$, then $G(k)_{x,0}, G(k)_{y,0} \supset G(k)_{z,0}$. In general, there exists $n \in N(k)$ such that $x$ and $nyn^{-1}$ lie in the closure of the same alcove, and then 
    $$G(k) = G(k)_{x,0} N(k) G(k)_{nyn^{-1},0} = G(k)_{x,0} N(k) n G(k)_{y,0} n^{-1} = G(k)_{x,0}N(k)G(k)_{y,0} n^{-1}$$
    so 
    $$G(k)=G(k)n = G(k)_{x,0}N(k)G(k)_{y,0}.$$
\end{proof}

\subsection{Epipelagic representations}\label{sec:epipelagic}

We briefly summarise the construction of epipelagic representations due to \cite{RY14}. Let $x$ be the barycentre of a facet in $\cA$. Write $V_x$ for the quotient $G(k^\ur)_{x,+:++}$. Then $V_x$ and its dual, $V_x^\vee$, admit algebraic representations of $\GG_x$ defined over $\FF_q$. Following \cite{Mum77}, we make the following definition from geometric invariant theory:

\begin{defn}
    We say that $\lambda \in V_x^\vee$ is a stable functional for the action of $\GG_x$ if the following hold:
    \begin{itemize}
        \item The orbit $\GG_x \cdot \lambda$ in $V_x^\vee$ is Zariski-closed.
        \item The stabiliser of $\lambda$ in $\GG_x$ is finite, as an algebraic group.
    \end{itemize}
\end{defn}

Motivated by the Hilbert--Mumford criterion for stability, we make the following definitions:

\begin{defn}
    Let $\sigma: G(k)_{x,+} \to \CC^\times$ be a character. We say that $\psi \in \Psi$ is in the support $\supp(\sigma)$ of $\sigma$ if $U_\psi \leq G(k)_{x,+}$ (equivalently, $\psi(x)>0$) and $\sigma(U_\psi) \neq 1$. In the notation of Section 2, we say that a tuple $(\psi_1,\dots,\psi_n)$ is in $\supp(\sigma)$ if $U_{\psi_i} \leq G(k)_{x,+}$ and $\sigma(U_{\psi_i}) \neq 1$ for each $\psi_i$. If $\sigma = \chi_\lambda$ for some nontrivial $\chi: \FF_q \to \CC^\times$ and $\lambda:G(k)_{x,+} \to \FF_q$ (as in the cases $\lambda \in V_x^\vee$ or $V(x)^\vee$), define $\supp(\lambda) = \supp(\chi_\lambda)$. As $\chi$ is nontrivial, this does not depend on the choice of $\chi$.
\end{defn}

Suppose now that $\lambda \in V_x^\vee(\FF_q) = (G(k)_{x,+:++})^\vee$. Denote by $G(k)_{x,\lambda}$ the stabiliser of $\chi_\lambda$ in $G(k)_{x}$ and define the intertwining algebra 
$$\cH_{\lambda} := \End_{G(k)_{x,\lambda}}\left( \ind_{G(k)_{x,+}}^{G(k)_{x,\lambda}} \chi_\lambda \right).$$ The subgroup $G(k)_{x,+}$ is normal in $G(k)_{x,\lambda}$, and let $A_\lambda$ denote the quotient, a finite group. If $\rho$ is a simple $\cH_\lambda$-module, let $\chi_{\lambda,\rho}$ denote the corresponding irreducible constituent of $\ind_{G(k)_{x,+}}^{G(k)_{x,\lambda}} \chi_\lambda$.

\begin{prop}\cite[Proposition 2.4]{RY14}\label{prop:RY}
    Suppose that $\lambda \in V_x^\vee(\FF_q)$ is an $\FF_q$-rational stable functional for the action of $\GG_x$. Then the following hold:
    \begin{enumerate}
        \item The representation $\pi_\lambda$ has a finite direct sum decomposition $$\pi_\lambda = \bigoplus\limits_{\rho \in \mathrm{Irr}(\cH_\lambda)} \dim \rho \cdot \pi(\lambda,\rho),$$ where $\pi(\lambda,\rho) := \ind_{G(k)_{x,\lambda}}^{G(k)} \chi_{\lambda,\rho}$ is an irreducible supercuspidal representation of $G(k)$, for each $\rho \in \mathrm{Irr}(\cH_\lambda)$.
        \item If $\rho,\rho'$ are inequivalent simple modules for $\cH_\lambda$, then $\pi(\lambda,\rho)$ and $\pi(\lambda,\rho')$ are inequivalent representations of $G(k)$.
        \item The formal degree of $\pi(\lambda,\rho)$ with respect to a Haar measure $\mu$ on $G(k)$ is given by 
        $$\mathrm{deg}_\mu(\pi(\lambda,\rho)) = \frac{\dim \chi_{\lambda,\rho}}{|A_\lambda|} \cdot \frac{1}{\mu(G(k)_{x,+})}.$$
    \end{enumerate}
\end{prop}

%We define the formal degree of a discrete series representation of $G(k)$ in Definition \ref{def:fdeg}. 

\begin{example}
	Let $x$ be the barycentre of $\cC$. Then $\GG_x = \TT_x$. As $T$ is split, we may identify $\Hom_k(T,\GG_m) \cong \Hom_{\FF_q}(\TT_x,\GG_m)$ as $\ZZ$-modules. Under this identification, the weight space decomposition for $V_x^\vee$ under the algebraic representation of $\GG_x =\TT_x$ is given by 
	$$V_x^\vee = \bigoplus\limits_{\psi \in \Pi} V_x^\vee(-\dot\psi)$$ where $V_x^\vee(-\dot\psi)$ are 1-dimensional vector spaces over $\overline \FF_q$ on which $\TT_x$ acts via the character $-\dot\psi \in \Phi$. We may identify $$V_x^\vee(-\dot\psi) \cong V_\psi^\vee \otimes_{\FF_q} \overline \FF_q$$ as representations of $\TT_x$. In this setting, $\lambda \in V_x^\vee(\FF_q)$ is stable if and only if $\lambda_{-\dot\psi} \in V_x^\vee(-\dot\psi)$ is nonzero for all $\psi \in \Pi$. This is the affine generic condition defined in \cite{GR10}, and the resulting supercuspidal representations are the simple supercuspidals. In particular, simple supercuspidals have minimal positive depth.
	% if zero for one \psi then take cocharacter to be \dot\psi^\check pairing non-positively with every other weight.
\end{example}

\begin{defn}[\cite{RY14}]
    An irreducible representation $\pi$ of $G(k)$ is epipelagic if, for some $x \in \cA$, $\pi$ has a nonzero fixed vector under $G(k)_{x,r(x)+}$ and has depth $r(x)$, where $r(x) \in \RR$ is defined by $G(k)_{x,r(x)}=G(k)_{x,+}$ and $G(k)_{x,r(x)+}=G(k)_{x,++}$.
\end{defn}

\begin{prop}\label{prop:epipelagic}
    Let $\lambda \in V_x^\vee(\FF_q)$ be an $\FF_q$-rational stable functional for the action of $\GG_x$. Then each of the irreducible supercuspidal representations $\pi(\lambda,\rho)$ constructed in Proposition \cite[Proposition 2.4]{RY14} is epipelagic.
\end{prop}
\begin{proof}
    By \cite[Theorem 4.2]{FR17}, given that such a stable functional exists, $x$ must be the barycentre of a facet. Conjugating as necessary, we assume $x \in \overline \cC$. Each $\chi_{\lambda,\rho}$ is $\chi_\lambda$-isotypic when restricted to $G(k)_{x,+}$. As $\chi_\lambda$ is trivial on $G(k)_{x,++}$ by definition, each $\pi(\lambda,\rho) = \ind_{G(k)_{x,\lambda}}^{G(k)} \chi_{\lambda,\rho}$ has a nonzero fixed vector under the action of $G(k)_{x,++}$, namely any function in $\ind_{G(k)_{x,\lambda}}^{G(k)} \chi_{\lambda,\rho}$ supported on $G(k)_{x,\lambda}$. It remains to show that $\pi(\lambda,\rho)$ has depth $r(x)$. 

    Suppose that $\pi_\lambda$ has a nonzero fixed vector $v$ under $G(k)_{y,r+}$ for some $y \in \cB$ and $r \geq 0$. As $g \cdot v$ is a nonzero fixed vector under $gG(k)_{y,r+}g^{-1} = G(k)_{g\cdot y,r+}$, we may assume after conjugating $y$ that $y \in \cA$. It suffices to show that $r\geq r(x)$. By Mackey theory we have 
    \begin{align*}
        0 &\neq \Hom_{G(k)_{y,r+}}(\mathbbm{1}_{G(k)_{y,r+}}, \pi_\lambda\mid_{G(k)_{y,r+}}) \\
        &\cong \bigoplus\limits_{g \in G(k)_{y,r+} \backslash G(k)/G(k)_{x,+}} \Hom_{G(k)_{y,r+}}(\mathbbm{1}_{G(k)_{y,r+}}, \ind_{\prescript{g}{}{G(k)_{x,+}} \cap G(k)_{y,r+}}^{G(k)_{y,r+}}  \prescript{g}{}{\chi_\lambda} \mid_{\prescript{g}{}{G(k)_{x,+}} \cap G(k)_{y,r+}}) \\
        &\cong \bigoplus\limits_{g \in G(k)_{y,r+} \backslash G(k)/G(k)_{x,+}} \Hom_{\prescript{g}{}{G(k)_{x,+}} \cap G(k)_{y,r+}}(\mathbbm{1}_{\prescript{g}{}{G(k)_{x,+}} \cap G(k)_{y,r+}}, \prescript{g}{}{\chi_\lambda} \mid_{\prescript{g}{}{G(k)_{x,+}} \cap G(k)_{y,r+}})
    \end{align*}
    where the last line follows from Frobenius reciprocity. The above sum being nonzero is equivalent to there existing $g$ such that $\prescript{g}{}{\chi_\lambda}$ is trivial on $\prescript{g}{}{G(k)_{x,+}} \cap G(k)_{y,r+}$. By the affine Bruhat decomposition, Proposition \ref{prop:bruhat}, we may write $g=anb$ for $a \in G(k)_{y,0}$, $b \in G(k)_{x,0}$ and $n \in N(k)$. Then $\prescript{b}{}{\chi_\lambda} = \chi_{b\cdot \lambda}$ is trivial on $$G(k)_{x,+} \cap G(k)_{y,r+}^{an} = G(k)_{x,+} \cap G(k)_{n\cdot y,r+}.$$
    Replacing $y$ by $n^{-1}\cdot y \in \cA$, and $\lambda$ by $b^{-1} \cdot \lambda$, which is still stable, this means that $\chi_\lambda$ is trivial on $G(k)_{x,+} \cap G(k)_{y,r+}$ for some $y \in \cA$. Recall that $\supp(\lambda)$ denotes the set of $\psi \in \Delta(x)$ such that $\lambda$ is not trivial on $U_\psi$. Then, in order for $\chi_\lambda$ to be trivial on $G(k)_{x,+} \cap G(k)_{y,r+}$, we must have that for all $\psi \in \supp(\lambda)$, we have $\psi(y) \leq r$. Otherwise, $U_\psi \subset G(k)_{x,+} \cap G(k)_{y,r+}$ and $\chi_\lambda$ is not trivial here. The assumption that $\lambda$ is stable implies, by the Hilbert--Mumford criterion, that for all nonzero $\gamma \in X_*(T) \otimes \RR = \cA$, there exists $\psi \in \supp(\lambda)$ such that $\langle \dot\psi, \gamma \rangle >0$.

    We have the following elementary lemma:

    \begin{lemma}\label{lem:pos_span}\cite{Dav54}
       Let $V$ be a finite-dimensional vector space over $\RR$ and suppose $f_1, \dots, f_n$ are linear functionals on $V$ with the property that for all nonzero $v \in V$, there exists $i$ such that $f_i(v)>0$. Then there exist real numbers $a_i \geq 0$, not all zero, such that $\sum_i a_i f_i = 0$.
    \end{lemma}
    \begin{proof}
        \cite[Theorem 3.1]{Dav54}.
    \end{proof}

    As a result, there exist $a_\psi \geq 0$, not all zero, such that 
    $$\sum\limits_{\psi \in \supp(\lambda)} a_\psi \psi$$ is a constant function. But then 
    $$r(x) \cdot \sum\limits_{\psi \in \supp(\lambda)} a_\psi = \sum\limits_{\psi \in \supp(\lambda)} a_\psi \psi(x) = \sum\limits_{\psi \in \supp(\lambda)} a_\psi \psi(y) \leq r \cdot \sum\limits_{\psi \in \supp(\lambda)} a_\psi$$
    so that $r \geq r(x)$.
\end{proof}

\subsection{Towards higher depth}

We would like to extend the construction of \cite{RY14} from stable functionals of $G(k)_{x,+:++}$ to certain functionals on $V(x)$. Unlike the case of $G(k)_{x,+:++}$, the space $V(x,k^\ur) = G(k^\ur)_{x,+:1}^\ab$ does not admit an algebraic representation of $\GG_x$, or even $\TT_x$. Indeed, there is different behaviour in $V(x)$ depending on whether $q=2,3$ or $q>3$, so that in general $$V(x,k^\ur) \not\cong V(x) \otimes_{\FF_q} \overline \FF_q.$$ 

By the Hilbert--Mumford criterion, a functional $\lambda \in V_x^\vee$ is stable if and only if, for any one-parameter subgroup $\gamma: \GG_m \to \GG_x$, $\lambda$ has a negative weight for the action of $\gamma$ on $V_x^\vee$ (induced by the action of $\GG_x$). In the construction of \cite{RY14}, it turns out that we only need this condition to hold for all $\gamma$ defined over $\FF_q$. In other words, we only require that $g\cdot \lambda$ has a negative weight for the action of $\gamma \in \Hom_{\FF_q}(\GG_m, \TT_x)$ for any $g \in \GG_x(\FF_q)$ and any $\gamma$. As $T$ is split over $k$, we may identify $$\Hom_{\FF_q}(\GG_m, \TT_x) = \Hom_k(\GG_m,T) = X_*(T).$$
If $\lambda$ is stable, then for any $\gamma \in X_*(T)$ and $g \in \GG_x(\FF_q)$, there exists $\psi \in \supp(g\cdot \lambda)$ such that $\pair{\dot\psi}{\gamma} >0$.

\begin{defn}\label{def:hm}
    Let $\lambda \in V(x)^\vee$. We say that $\lambda$ is $\FF_q$-stable if, for any $\gamma \in X_*(T)$ and $g \in \GG_x(\FF_q)$, there exists $\psi \in \supp(g\cdot \lambda)$ such that $\pair{\dot\psi}{\gamma} >0$.
\end{defn}

\begin{rem}
    At least when $q=2$, this is a strictly weaker condition than stability as defined in the previous subsection. This means that one could use the same method as Reeder--Yu to produce supercuspidal representations coming from vectors in $V_x^\vee$ that may be $\FF_q$-stable but not stable. For an example of this, see Section 4.2. The main observation in this example is that the adjoint representation of $\SL_2$ over $\overline \FF_2$ has no stable vectors (this is true of any adjoint representation) but, identifying the adjoint representation with the representation on degree 2 homogeneous polynomials in $X,Y$ over $\overline \FF_2$, the action of $\SL_2(\FF_2)$ on $X^2+XY+Y^2$ is trivial. In Section 4.2, for $G$ of type $B_n$, we consider a point $y$ in the apartment with $\GG_y(\FF_2)\cong \SL_2(\FF_2)$, from which we produce an $\FF_2$-stable vector that is not stable.
\end{rem}

\begin{lemma}\label{lem:lin_alg}
    Let $V$ be a finite-dimensional vector space over $\RR$ and suppose $f_1, \dots, f_n$ are linear functionals on $V$ with the property that for all nonzero $v \in V$, there exists $i$ such that $f_i(v)>0$. Then for any $r>0$
    $$\{v \in V: f_i(v) \leq r \text{ for all } 1 \leq i \leq n\}$$ is compact.
\end{lemma}
\begin{proof}
    Let $S$ denote the unit sphere in $V$ and $B$ the unit disc so that the boundary of $B$ is $S$. For each $v \in S$, there exists $i$ such that $f_i(v)>0$. Therefore there exists $R(v) >0$ such that $$\max_i f_i(R(v)v) = r$$ and $R(v)$ is continuous in $v$. As $S$ is compact, the function $R$ takes a maximum value $R_0$ on $S$. Then we have the containment 
    $$\{v \in V: f_i(v) \leq r \text{ for all } 1 \leq i \leq n\} \subset R_0\cdot B$$ so that $\{v \in V: f_i(v) \leq r \text{ for all } 1 \leq i \leq n\}$ is closed and bounded and therefore compact.
\end{proof}

\begin{prop}\label{prop:finite_intertwine}
    Let $\lambda \in V(x)^\vee$ be an $\FF_q$-stable functional. Fix a positive real number $r$. There exist finitely many $g \in G(k)_{x,r} \backslash G(k) / G(k)_{x,+}$ such that $\lambda$ vanishes on $G(k)_{x,+} \cap g^{-1}G(k)_{x,r}g$.
\end{prop}
\begin{proof}
    By Proposition \ref{prop:bruhat} a set of $G(k)_{x,0}$ double coset representatives of $G(k)$ is contained in a set of lifts of elements of $W(G)$ to $N(k)$. Write $g=anb$ where $a,b \in G(k)_{x,0}$ and $n \in N(k)$ has image $w\in W(G)$. Then 
    $$G(k)_{x,+} \cap g^{-1}G(k)_{x,r}g = G(k)_{x,+} \cap b^{-1}n^{-1}G(k)_{x,r}nb = G(k)_{x,+} \cap b^{-1}G(k)_{y,r}b$$
    where $y=w^{-1}(x) \in \cA(T)$ and the first equality follows from the fact that $G(k)_{x,r}$ is normal in $G(k)_{x,0}$. So we need to show that there are finitely many pairs of $w\in W(G)$ and $b \in G(k)_{x,0}/G(k)_{x,+} \cong \GG_x(\FF_q)$ such that $\lambda$ vanishes on $G(k)_{x,+} \cap b^{-1} G(k)_{y,r} b$, or equivalently $b \cdot \lambda$ vanishes on $G(k)_{x,+} \cap G(k)_{y,r}$. By definition of stability, $b \cdot \lambda$ is still $\FF_q$-stable. As there are finitely many choices for $b$, we may fix $b$ and show there are only finitely many $w$ with the above property.

    Set $\gamma = y-x \in X_*(T) \otimes \RR$. We have for any $\psi \in \Psi$
    $$\psi(y) = \psi(x+\gamma) = \psi(x)+\pair{\dot\psi}{\gamma}.$$
    Observe that $G(k)_{x,+}\cap G(k)_{y,r}$ contains $U_\psi$ for all $\psi \in \Psi$ such that $\psi(x)>0$ and $\psi(y) \geq r$. If $\psi \in \supp(b\cdot \lambda)$ then we automatically have $\psi(x)>0$. So, for such $\psi$, if $\pair{\dot\psi}{\gamma} \geq r$, we certainly have $\psi(y) \geq r$. It suffices to show that there exist finitely many $w \in W(G)$ such that for $\gamma = w^{-1}(x)-x$ we have $\pair{\dot\psi}{\gamma} \leq r$ for all $\psi \in \supp(b\cdot \lambda)$. This follows from Lemma \ref{lem:lin_alg} and the fact that under the action of $W(G)$ on $X_*(T) \otimes \RR$, the stabiliser of any point is finite and the orbit of any point is discrete.
\end{proof}

\begin{cor}
    Let $\lambda \in V(x)^\vee$ be an $\FF_q$-stable functional. Then $I(G,G(k)_{x,+},\chi_\lambda)$ is a union of finitely many $G(k)_{x,+}$-double cosets.
\end{cor}
\begin{proof}
    This follows from the above proposition when taking $r$ such that $G(k)_{x,r}=G(k)_{x,+}$.
\end{proof}

\begin{cor}\label{cor:admissible}
    Let $\lambda \in V(x)^\vee$ be an $\FF_q$-stable functional. The compactly induced representation $(\pi_\lambda,\Sigma) = \ind_{G(k)_{x,+}}^{G(k)} \chi_\lambda$ is admissible.
\end{cor}
\begin{proof}
    We must show that for any compact open subgroup $K$ of $G(k)$, the vector space of fixed points $\Sigma^K$ of $K$ is finite dimensional. We may shrink $K$ to assume that $K=G(k)_{x,r}$ for some positive real number $r$, as the Moy--Prasad filtration at any $x$ is a neighbourhood basis of the identity. Then we must show that $\Hom_{G(k)_{x,r}}(\mathbbm{1}_{G(k)_{x,r}}, \pi_\lambda\mid_{G(k)_{x,r}})$ is finite dimensional. By Frobenius reciprocity and Mackey theory we have 
    $$\Hom_{G(k)_{x,r}}(\mathbbm{1}_{G(k)_{x,r}}, \pi_\lambda\mid_{G(k)_{x,r}}) \cong \bigoplus\limits_{g \in G(k)_{x,r} \backslash G(k)/G(k)_{x,+}} \Hom_{G(k)_{x,+} \cap G(k)_{x,r}^g}(\mathbbm{1}_{G(k)_{x,+} \cap G(k)_{x,r}^g}, \chi_\lambda\mid_{G(k)_{x,+} \cap G(k)_{x,r}^g}).$$
    As $\chi_\lambda$ is a character, and $\chi$ is nontrivial on $\FF_q$, $\Hom_{G(k)_{x,+} \cap G(k)_{x,r}^g}(\mathbbm{1}_{G(k)_{x,+} \cap G(k)_{x,r}^g}, \chi_\lambda\mid_{G(k)_{x,+} \cap G(k)_{x,r}^g})$ is nonzero (and 1-dimensional) if and only if $\lambda$ vanishes on $G(k)_{x,+} \cap G(k)_{x,r}^g$. The result follows from Proposition \ref{prop:finite_intertwine}.
\end{proof}

\begin{thm}\label{thm:main}
    Let $\lambda \in V(x)^\vee$ be an $\FF_q$-stable functional. The compactly induced representation $(\pi_\lambda,\Sigma) = \ind_{G(k)_{x,+}}^{G(k)} \chi_\lambda$ is a finite direct sum of irreducible supercuspidal representations of $G(k)$.
\end{thm}
\begin{proof}
    This follows immediately from Corollary \ref{cor:admissible} and \cite[Theorem 1]{Bus90}.
\end{proof}

% idea is that admissibility implies supercuspidality of the compactly induced representation - in terms of matrix coefficients, where the dual (which is the smooth induction of the dual rep) is actually the compact induction of the dual rep. Then result is formal from projectivity ? of supercuspidals

Finally, we describe some techniques for computing intertwining sets.

\begin{lemma}\label{lem:weylsupport}
    Let $n\in N(k)$ with image $w \in W(G)$ and let $a,b \in G(k)_{x,0}$. Let $\lambda \in V(x)^\vee$. If $a^{-1}nb \in I(G,G(k)_{x,+},\chi_\lambda)$, then 
    $$w(\supp(b\cdot \lambda)) \cap \Psi_x^+ \subset \supp(a \cdot \lambda).$$
    %where $\Psi_x^{\leq 0}$ denotes the affine roots taking non-positive value at $x$. 
\end{lemma}
\begin{proof}
    Suppose that $\psi \in \supp(b\cdot \lambda)= \supp(\chi_\lambda^{b^{-1}})$ and $w(\psi) \in \Psi_x^+$ but $w(\psi) \not\in \supp(a \cdot \lambda) = \supp(\chi_\lambda^{a^{-1}})$. If $a^{-1}nb$ intertwines $\chi_\lambda$, then $\chi_\lambda=\chi_\lambda^{a^{-1}nb}$ on $G(k)_{x,+} \cap G(k)_{x,+}^{a^{-1}nb}$. So $\chi_\lambda^{b^{-1}}=\chi_\lambda^{a^{-1}n}$ on $G(k)_{x,+}^{b^{-1}} \cap G(k)_{x,+}^{a^{-1}n} = G(k)_{x,+} \cap G(k)_{x,+}^n$. This equality follows from the fact that $G(k)_{x,+}$ is normal in $G(k)_{x,0}$. Since $U_\psi = U_{w(\psi)}^n$, and $w(\psi) \in \Psi_x^+$, we have $U_\psi \leq G(k)_{x,+} \cap G(k)_{x,+}^n$. By assumption, $\chi_\lambda^{a^{-1}}$ is trivial on $U_{w(\psi)}$, so $\chi_\lambda^{b^{-1}}=\chi_\lambda^{a^{-1}n}$ is trivial on $U_\psi=U_{w(\psi)}^{n}$, a contradiction to $\psi \in \supp(\chi_\lambda^{b^{-1}})$. 
\end{proof}

\begin{lemma}\label{lem:zeros}
    Let $n\in N(k)$ with image $w \in W(G)$ and let $a,b \in G(k)_{x,0}$. Let $\lambda \in V(x)^\vee$ be an $\FF_q$-stable functional. Let $\cF$ be the facet of $\cC$ of which $x$ is the barycentre. If $a^{-1}nb \in I(G,G(k)_{x,+},\chi_\lambda)$, then there does not exist $y \in \overline\cF$ such that $\theta(y)=0$ for all $\theta \in w(\supp(b \cdot \lambda)) \cap \Psi_x^+$.
\end{lemma}
\begin{proof}
    Suppose such a $y$ existed. Then for all $\psi \in \supp(b\cdot \lambda)$, we have that $w(\psi) \not\in \Psi_x^+$ (and so $w(\psi)$ is non-positive on $\cF$) or $\psi(w^{-1}y)=w(\psi)(y)=0$. By definition, $b\cdot \lambda$ is still $\FF_q$-stable. By Lemma \ref{lem:pos_span}, there exist real numbers $a_\psi \geq 0$ such that 
    $$\sum\limits_{\psi \in \supp(b\cdot \lambda)} a_\psi \psi \text{ is a constant function $c$.}$$ As each $\psi$ is positive on $x$, this constant $c$ must be positive. But then 
    $$c = \sum\limits_{\psi\in\supp(b\cdot \lambda)} a_\psi \psi(w^{-1}y) = \sum\limits_{\psi \in \supp(b\cdot \lambda)} a_\psi w(\psi)(y) \leq 0,$$ a contradiction.
\end{proof}

\section{Examples}

\subsection{$G=G_2$, $q=2$}

We take $q$ to be 2 and $G$ to be of type $G_2$ with affine Dynkin diagram labelled as below:

%%%%%%%%%%%%%%%%%%%% TYPE G2 %%%%%%%%%%%%%%%%%%%%

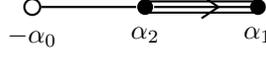
\begin{center}

\begin{tikzpicture}[
  scale=1.5,
  vertex/.style={fill=black, circle, inner sep=0pt, minimum size=6pt},
  empty vertex/.style={circle, draw=black, inner sep=0pt, minimum size=6pt},
  every path/.style={thick},
  label style/.style={below=4pt},
]

\coordinate (0) at (0,0);
\coordinate (1) at (1,0);        % alpha_2 center
\coordinate (2) at (2,0);     % alpha_3 to the right

% Nodes with labels
\node[empty vertex] (0) at (0) {};
\node[label style] at (0) {$-\alpha_0$};

\node[vertex] (1) at (1) {};
\node[label style] at (1) {$\alpha_2$};

\node[vertex] (2) at (2) {};
\node[label style] at (2) {$\alpha_1$};

% Edges - uniform length ~1 unit

\draw (0) -- (1);

% Triple bond lines close together
\draw (1)++(0,0.05) -- ++(1,0);
\draw (1)++(0,-0.05) -- ++(1,0);
\draw (1) -- ++(1,0);

\coordinate (Vmid2) at ($(1)!0.5!(2)$);
\begin{scope}[shift={(Vmid2)}, rotate=0]
  \draw[thick] (0,0.1) -- (0.15,0) -- (0,-0.1);
\end{scope}
\end{tikzpicture}
\end{center}

Let $\omega_1^\vee, \omega_2^\vee$ be the fundamental coweights. Set $x$ to be the point with Kac coordinates $11\Rrightarrow 1$ (the barycentre of the alcove $\cC$), so that $G(k)_{x,0}$ is an Iwahori subgroup of $G(k)$. The reductive quotient $\GG_x$ is a split rank 2 torus over $\FF_2$. In particular, $\GG_x(\FF_2)=1$ and so $G(k)_{x,+}=G(k)_{x,0}$. We set 
$$\cS(x) = \{1-\alpha_0,\alpha_1,\alpha_2,(\alpha_1+\alpha_2,2\alpha_1+\alpha_2)\}$$
and define the $\FF_2$-vector space $\cV(x)$ in the same way as in Section 2. Then the abelianisation $V(x)$ of $G(k)_{x,+:1}$ is isomorphic to $\cV(x)$ as an abelian group - see the author's forthcoming thesis for further details.

Let $\lambda \in V(x)^\vee$ be a functional with support containing $\{1-\alpha_0,\alpha_1+\alpha_2,2\alpha_1+\alpha_2\}$. Then $\lambda$ is an $\FF_2$-stable functional. Indeed the action of $\GG_x(\FF_2)$ is trivial, and if $\gamma = r_1\omega_1^\vee+r_2\omega_2^\vee$ is a cocharacter of $\TT_x$, where $r_i \in \ZZ$, such that 
$$\pair{\dot\psi}{\gamma} \leq 0 \text{ for all $\psi \in \{1-\alpha_0,\alpha_1+\alpha_2,2\alpha_1+\alpha_2\}$},$$ then we have the following inequalities:
\begin{align*}
    \pair{-\alpha_0}{\gamma} \leq 0 &\iff 3r_1+2r_2 \geq 0 \\
    \pair{\alpha_1+\alpha_2}{\gamma} \leq 0 &\iff r_1+r_2 \leq 0 \\
    \pair{2\alpha_1+\alpha_2}{\gamma} \leq 0 &\iff 2r_1+r_2 \leq 0
\end{align*}
These inequalities together imply $r_1=r_2=0$, as required.

Let $\lambda_1,\lambda_2$ be the two functionals in $V(x)^\vee$ with support containing $\{1-\alpha_0,\alpha_2,\alpha_1+\alpha_2,2\alpha_1+\alpha_2\}$.

\begin{prop}
    The compactly induced representations $$\pi_i := \ind_{G(k)_{x,+}}^{G(k)} \chi_{\lambda_i}$$ are irreducible and supercuspidal.
\end{prop}
\begin{proof}
As $G(k)_{x,+}=G(k)_{x,0}$, it suffices to check that the only element $w$ of the affine Weyl group that intertwines $\chi_{\lambda_i}$ is trivial. By Lemma \ref{lem:weylsupport}, 
$$w(\{1-\alpha_0,\alpha_2,\alpha_1+\alpha_2,2\alpha_1+\alpha_2\}) \cap \Psi_x^+ \subset \{1-\alpha_0,\alpha_1,\alpha_2,\alpha_1+\alpha_2,2\alpha_1+\alpha_2\}.$$
We must also have 
$$w(1-\alpha_0)+w(\alpha_1+\alpha_2)+w(2\alpha_1+\alpha_2)=1$$ as $\alpha_0=3\alpha_1+2\alpha_2$. As the simple affine roots take value $1/6$ at $x$, we see that the above conditions are only compatible when 
$$w(\{1-\alpha_0,\alpha_1+\alpha_2,2\alpha_1+\alpha_2\}) = \{1-\alpha_0,\alpha_1+\alpha_2,2\alpha_1+\alpha_2\},$$
and so $w$ is either trivial or the simple reflection $w_{\alpha_1}$ in $\alpha_1=0$. The latter sends $\alpha_2$ to $3\alpha_1+\alpha_2 \in \Psi_x^+$, violating the condition from Lemma \ref{lem:weylsupport}. Therefore $w$ must be trivial. Lemma \ref{lem:mackey_basic} implies that the $\pi_i$ are irreducible and supercuspidal for $i=1,2$.
\end{proof}

However, 

\begin{prop}
    Let $y \in \overline \cC$ be the point with Kac coordinates $11 \Rrightarrow 0$, so that $\cS(y)=\Delta(y) = \{1-\alpha_0,\alpha_2,\alpha_1+\alpha_2,2\alpha_1+\alpha_2,3\alpha_1+\alpha_2\}$. Let $\eta \in V_y^\vee(\FF_2)$ be the functional with support $\{1-\alpha_0,\alpha_2,\alpha_1+\alpha_2,2\alpha_1+\alpha_2\}$. This is a stable functional. The representations $\pi_i$ are the two irreducible constituents of the representation $\pi_\eta$ constructed by \cite[Proposition 2.4]{RY14}.
\end{prop}
\begin{proof}
    The fact that $\eta$ is stable follows from \cite[Section 7.5]{RY14} or \cite[Section 4.2]{Rom16}. The fact that $\pi_\eta$ has two irreducible constituents follows from \cite[Section 7.5, Proposition 2.4]{RY14}. It suffices to show that $\pi_1,\pi_2$ are non-isomorphic, and that 
    $$\Hom_{G(k)}(\ind_{G(k)_{x,+}}^{G(k)}\chi_{\lambda_i}, \ind_{G(k)_{y,+}}^{G(k)} \chi_\eta) \neq 0.$$
    By Frobenius reciprocity and Mackey theory, the left hand side is isomorphic to 
    $$
    \bigoplus\limits_{g \in G(k)_{x,+} \backslash G(k) / G(k)_{y,+}}\Hom_{G(k)_{x,+} \cap \prescript{g}{}{G(k)_{y,+}}}(\chi_{\lambda_i} \mid_{G(k)_{x,+} \cap \prescript{g}{}{G(k)_{y,+}}}, \prescript{g}{}{\chi_\eta}\mid_{G(k)_{x,+} \cap \prescript{g}{}{G(k)_{y,+}}} ).
    $$
    We have $G(k)_{x,+} \cap G(k)_{y,+}=G(k)_{y,+}$. As $U_{\alpha_1} \not\leq G(k)_{y,+}$, we see that $\lambda_1,\lambda_2$ both agree with $\eta$ on $G(k)_{y,+}$, so that the above term is nonzero. Finally, we show that $\pi_1 \not\cong \pi_2$. If not, then
    $$\Hom_{G(k)}(\ind_{G(k)_{x,+}}^{G(k)}\chi_{\lambda_1}, \ind_{G(k)_{x,+}}^{G(k)} \chi_{\lambda_2}) \neq 0.$$
    The same argument, together with the affine Bruhat decomposition and the fact that $G(k)_{x,+}=G(k)_{x,0}$, implies that there exists $w\in W(G)$ such that $\chi_{\lambda_1},\prescript{w}{}{\chi_{\lambda_2}}$ agree on $G(k)_{x,+} \cap \prescript{w}{}{G(k)_{x,+}}$. Suppose $\psi \in \supp(\lambda_1)$. If $w^{-1}(\psi) \in \Psi_x^+$, then $U_{\psi} = \prescript{w}{}{U_{w^{-1}(\psi)}} \leq G(k)_{x,+} \cap \prescript{w}{}{G(k)_{x,+}}$. So we must have $w^{-1}(\psi) \in \supp(\lambda_2)$. In other words, $$w^{-1}(\supp(\lambda_1)) \cap \Psi_x^+ \subset \supp(\lambda_2).$$
    As in the previous proposition, this implies that $w$ is trivial. But $\lambda_1 \neq \lambda_2$, so the $\pi_i$ are distinct.
\end{proof}

\subsection{$G=B_n$, $q=2$}

We take $q$ to be 2 and $G$ to be of type $B_n$ with affine Dynkin diagram labelled as follows:

%%%%%%%%%%%%%%%% TYPE BN %%%%%%%%%%%%%%%%%%%%%%%

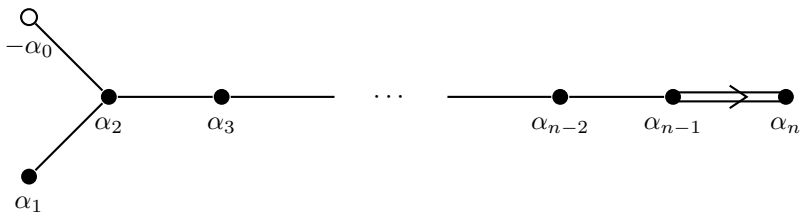
\begin{center}

\begin{tikzpicture}[
  scale=1.5,
  vertex/.style={fill=black, circle, inner sep=0pt, minimum size=6pt},
  empty vertex/.style={circle, draw=black, inner sep=0pt, minimum size=6pt},
  every path/.style={thick},
  label style/.style={below=4pt},
]

% Coordinates: symmetric above at 135 and 225 degrees
\coordinate (2) at (0,0);        % alpha_2 center
\coordinate (0) at ($(2)+(135:1)$);   % -alpha_0 at 135°
\coordinate (1) at ($(2)+(225:1)$);   % alpha_1 at 225°
\coordinate (3) at ($(2)+(1,0)$);     % alpha_3 to the right
\coordinate (dots) at ($(3)!0.5!(4,0)$); % ellipsis between alpha_3 and alpha_{n-1}
\coordinate (n-2) at (4,0);
\coordinate (n-1) at (5,0);
\coordinate (n) at (6,0);
\coordinate (arrowtip) at ($(n-1)!0.5!(n)$);

% Nodes with labels
\node[empty vertex] (0) at (0) {};
\node[label style] at (0) {$-\alpha_0$};

\node[vertex] (1) at (1) {};
\node[label style] at (1) {$\alpha_1$};

\node[vertex] (2) at (2) {};
\node[label style] at (2) {$\alpha_2$};

\node[vertex] (3) at (3) {};
\node[label style] at (3) {$\alpha_3$};

\node[vertex] (n-2) at (n-2) {};
\node[label style] at (n-2) {$\alpha_{n-2}$};

\node[vertex] (n-1) at (n-1) {};
\node[label style] at (n-1) {$\alpha_{n-1}$};

\node[vertex] (n) at (n) {};
\node[label style] at (n) {$\alpha_n$};

% Edges - uniform length ~1 unit
\draw (0) -- (2);
\draw (1) -- (2);
\draw (2) -- (3);
\draw (3) -- ++(1,0);         % partial edge toward ellipsis start
\draw (n-2) -- ++(-1,0);      % partial edge from ellipsis end
\draw (n-2) -- (n-1);

% Ellipsis centered properly
\node at (dots) {$\cdots$};

% Double bond lines close together
\draw (n-1)++(0,0.04) -- ++(1,0);
\draw (n-1)++(0,-0.04) -- ++(1,0);

% Custom ">" shaped arrowhead overlaid on double bond midpoint
\coordinate (Vmid) at (arrowtip);
\coordinate (Vtip) at ($(Vmid)+(0.15,0)$);
\coordinate (Vtop) at ($(Vmid)+(0,0.1)$);
\coordinate (Vbottom) at ($(Vmid)+(0,-0.1)$);
\draw[thick] (Vtop) -- (Vtip) -- (Vbottom);

\end{tikzpicture}
\end{center}

We have $\alpha_0 = \alpha_1+2(\alpha_2+\dots \alpha_{n})$. We take $x$ to be the barycentre of $\cC$. We find that 
$$V(x) \cong \bigoplus\limits_{\psi \in \cS(x)}V_\psi$$
for $\cS(x)= \Pi \sqcup \{(\alpha_{n-1}+\alpha_{n},\alpha_{n-1}+2\alpha_{n})\}$. Take $\lambda \in V(x)^\vee$ to be the unique $\FF_2$-linear functional with support $\cS(x)$. Then $\lambda$ is an $\FF_2$-stable functional. Indeed the action of $\GG_x(\FF_2)$ is trivial, and if $\gamma = \sum_i r_i\omega_i^\vee$ is a cocharacter of $\TT_x$, where $r_i \in \ZZ$ and $\omega_i^\vee$ are the fundamental coweights, such that 
$$\pair{\dot\psi}{\gamma} \leq 0 \text{ for all $\psi \in \Pi \sqcup \{\alpha_{n-1}+\alpha_{n},\alpha_{n-1}+2\alpha_{n}\}$},$$ then we have the following inequalities:
\begin{align*}
    \pair{-\alpha_0}{\gamma} \leq 0 &\iff r_1+2(r_2+\dots r_{n}) \geq 0 \\
    \pair{\alpha_i}{\gamma} \leq 0 &\iff r_i \leq 0 \text{ for each $i$ } \\
\end{align*}
These inequalities imply $r_i=0$ for all $i$, as required. Then, by Theorem \ref{thm:main}, letting $\chi$ be the unique nontrivial additive character $\chi:  \FF_2 \to \CC^\times$, the compactly induced representation $\pi_\lambda := \ind_{G(k)_{x,+}}^{G(k)} \chi_\lambda$ is a finite direct sum of irreducible supercuspidal representations of $G(k)$.

\begin{rem}
    A similar construction works for split groups of type $C_n$.
\end{rem}

\begin{prop}\label{prop:Bn}
    Each irreducible constituent of $\pi_\lambda$ has depth at least $\frac{1}{2(n-1)}$. At least one of them is epipelagic of depth $\frac{1}{2(n-1)}$.
\end{prop}
\begin{proof}
    By the proof of Proposition \ref{prop:epipelagic}, if $\pi_\lambda$ has a nontrivial fixed vector under $G(k)_{y',r+}$ for some $y' \in \cB$, then there exists $y \in \cA$ and $g \in \GG_x(\FF_q)$ such that $g\cdot \lambda$ is trivial on $G(k)_{x,+} \cap G(k)_{y,r+}$. In our case $\GG_x(\FF_2)$ is trivial. Let $y = \sum_i r_i \omega_i^\vee$ for $r_i \in \RR$. Then $\lambda$ is trivial on $G(k)_{x,+} \cap G(k)_{y,r+}$ if and only if the following inequalities hold:
$$
    1-(r_1+2(r_2+\dots r_{n})) \leq r, \hspace{1em}
    r_i \leq r \text{ for each $i$, } 
$$
$$
    r_{n-1}+r_{n} \leq r, \hspace{1em}
    r_{n-1}+2r_{n} \leq r.
$$
That $r \geq \frac{1}{2(n-1)}$ follows from the equation 
$$1-(r_1+2(r_2+\dots r_{n})) + r_1 + 2r_2+\dots +2r_{n-2} + 2(r_{n-1}+r_n) = 1.$$
When $r=\frac{1}{2(n-1)}$, a solution to the above inequalities is given by $r_1=\dots=r_{n-2}=r_{n-1} = \frac{1}{2(n-1)}$ and $r_n=0$. Let $y \in \overline \cC$ be the point with $\alpha_1(y)=\dots=\alpha_{n-1}(y)=\frac{1}{2(n-1)}$ and $\alpha_n(y)=0$. Then $\lambda$ is trivial on $G(k)_{y,\frac{1}{2(n-1)}+}$ and so $\pi_\lambda$ has a nonzero fixed vector under $G(k)_{y,\frac{1}{2(n-1)}+}$. Therefore, at least one irreducible constituent of $\pi_\lambda$ is epipelagic of depth $\frac{1}{2(n-1)}$.
\end{proof}

\begin{cor}\label{cor:Bn}
    When $n \geq 3$, at least one irreducible constituent of $\pi_\lambda$ does not arise from the construction of \cite{RY14}.
\end{cor}
\begin{proof}
    Combining \cite[Theorem 4.1]{RY14} and \cite[Section 8.2.2]{RLYG12}, the points $x \in \cA$ at which $V_x^\vee$ admits an $\FF_q$-rational stable functional satisfy $\delta(x) = m/2n$, where $m \mid n$. By Proposition \ref{prop:epipelagic}, the corresponding supercuspidal representations constructed by \cite[Proposition 2.4]{RY14} have depth $m/2n$. When $n \geq 3$ this does not take the value $\frac{1}{2(n-1)}$.
\end{proof}

We also compute the intertwining set $I(G,G(k)_{x,+},\chi_\lambda)$. Let $\sigma \in W(\Psi)$ be the unique order 2 element of the extended affine Weyl group that stabilises $\cC$. The action on $\Pi$ is given by the unique order 2 symmetry of the affine Dynkin diagram of type $B_n$. We have that $\sigma \in W(G)$ if $G = \mathrm{Spin}_{2n+1}$ (so $G$ is simply connected) and not if $G=\mathrm{SO}_{2n+1}$ (so $G$ is the adjoint group). In particular, $\sigma \in G(k)_{x,0}=G(k)_{x,+}$ if $G = \mathrm{Spin}_{2n+1}$ and not if $G=\mathrm{SO}_{2n+1}$. Let $w_{\alpha_n}$ denote the simple reflection in $W(G)$ corresponding to $\alpha_n$. Then $\sigma w_{\alpha_n}=w_{\alpha_n}\sigma$. 

\begin{prop}\label{prop:BnI}
    Let $n \geq 5$. The intertwining set $I(G,G(k)_{x,+},\chi_\lambda)$ is the union of four (not necessarily disjoint) $G(k)_{x,+}$-double cosets represented by lifts to $N(k)$ of $1,w_{\alpha_n}, \sigma$ and $\sigma w_{\alpha_n}$.
\end{prop}
\begin{proof}
    As we are working in the setting $q=2$, we have that $G(k)_{x,+}=G(k)_{x,0}$. Note that in the simply-connected case, the $G(k)_{x,+}$-double cosets of $1$ and $\sigma$ agree, and also of $w_{\alpha_n}$ and $\sigma w_{\alpha_n}$ agree. By the affine Bruhat decomposition \ref{prop:bruhat}, it suffices to show that the elements of $W(G)$ that intertwine $\chi_\lambda$ are $1,w_{\alpha_n}, \sigma$ and $\sigma w_{\alpha_n}$. These certainly intertwine as their action on $\Psi$ takes $\supp(\lambda)$ to itself so that $\lambda=\lambda^w$ on $G(k)_{x,+}\cap G(k)_{x,+}^w$ for any of $w= 1,w_{\alpha_n}, \sigma$ and $\sigma w_{\alpha_n}$.

    On the other hand, by Lemma \ref{lem:weylsupport}, if $w \in W(G)$ intertwines $\chi_\lambda$, then $$w(\supp(\lambda)) \cap \Psi_x^+ \subset \supp(\lambda), \text{ where } \supp(\lambda) = \cS(x).$$
    Moreover, by Lemma \ref{lem:zeros}, the set $w(\cS(x)) \cap \cS(x)$ has no zeros in $\overline \cC$. This set must therefore contain $1-\alpha_0,\alpha_1,\dots,\alpha_{n-2}$ and also $\{\alpha_{n-1}+\alpha_n\}$, $\{\alpha_{n-1},\alpha_n\}$, $\{\alpha_{n-1},\alpha_{n-1}+2\alpha_n\}$ or $\{\alpha_{n},\alpha_{n-1}+2\alpha_n\}$. The only two elements of $\cS(x)$ whose gradients pair negatively with the gradients of at least three other elements of $\cS(x)$ are $\alpha_2$ and $\alpha_{n-2}$. For $\alpha_2$, these three elements are $1-\alpha_0,\alpha_1$ and $\alpha_3$. As these are all in $w(\cS(x)) \cap \cS(x)$ (as $n \geq 5$), it follows that $w(\alpha_2)=\alpha_2$ or $w(\alpha_{n-2})=\alpha_2$. Here we use the fact that the Weyl group acts by isometries on the roots. 
    
    The three affine roots in $\cS(x)$ whose gradients are long roots that pair negatively with $\alpha_{n-2}$ are $\alpha_{n-3}$, $\alpha_{n-1}$, $\alpha_{n-1}+2\alpha_n$. In the case that $w(\alpha_{n-2})=\alpha_2$, we must have, again by the fact that the Weyl group acts by isometries, that 
    $$w(\{\alpha_{n-3}, \alpha_{n-1}, \alpha_{n-1}+2\alpha_n\}) = \{1-\alpha_0,\alpha_1,\alpha_3\}.$$
    But $\alpha_{n-1}+2\alpha_n - \alpha_{n-1}$ is twice an affine root, while the difference of any two of $1-\alpha_0,\alpha_1,\alpha_3$ is not twice an affine root, so such $w$ does not exist.

    Given $w(\alpha_2)=\alpha_2$, again using combinatorics of the affine Dynkin diagram, $w$ must fix $\alpha_3,\dots, \alpha_{n-2}$ and either fix or swap $1-\alpha_0,\alpha_1$. Composing with $\sigma$, which swaps $1-\alpha_0,\alpha_1$, we may assume that $w$ fixes them. We must show that $w=1$ or $w=w_{\alpha_n}$. Out of the remaining elements $\alpha_{n-1}$, $\alpha_n$, $\alpha_{n-1}+\alpha_n$ and $\alpha_{n-1}=2\alpha_n$, the only one to pair with $\alpha_{n-2}$ to 0 is $\alpha_n$. So if $\alpha_n \in w(\cS(x)) \cap \cS(x)$ then $w(\alpha_n)=\alpha_n$ and the only possibility is $w=1$. Similarly, $\alpha_{n-1}+\alpha_n$ is the only short root of the four that pairs negatively with $\alpha_{n-2}$. If $\alpha_{n-1}+\alpha_n \in w(\cS(x))\cap \cS(x)$ then $w(\alpha_{n-1}+\alpha_n)=\alpha_{n-1}+\alpha_n$. One of $\alpha_{n-1}$ or $\alpha_{n-1}+2\alpha_n$ must then map under $w$ to a positive affine root at $x$, necessarily in $\cS(x)$ by Lemma \ref{lem:weylsupport}. The only possibility is that $$w(\{\alpha_{n-1},\alpha_{n-1}+2\alpha_n\})=\{\alpha_{n-1},\alpha_{n-1}+2\alpha_n\}.$$
    This gives $w=1$ or $w=w_{\alpha_{n}}$. In the last case that $$w(\cS(x))\cap \cS(x) = \{1-\alpha_0,\alpha_1,\dots,\alpha_{n-2},\alpha_{n-1},\alpha_{n-1}+2\alpha_n\}$$ we similarly find that $$w(\{\alpha_{n-1},\alpha_{n-1}+2\alpha_n\})=\{\alpha_{n-1},\alpha_{n-1}+2\alpha_n\}$$
    and so $w=1$ or $w=w_{\alpha_n}$. This exhausts all cases and completes the proof.
\end{proof}

We have seen that $\pi_\lambda$ is a direct sum of irreducible supercuspidal representations of $G(k)$. In fact we have:

\begin{cor}
    If $G=\mathrm{Spin}_{2n+1}$ with $n \geq 5$, then $\pi_\lambda$ is a direct sum of two non-isomorphic irreducible supercuspidal representations of $G(k)$. If $G=\mathrm{SO}_{2n+1}$ with $n \geq 5$, then $\pi_\lambda$ is a direct sum of four non-isomorphic irreducible supercuspidal representations of $G(k)$.
\end{cor}
\begin{proof}
    By combining \cite[Lemma 11.2]{BH1} and \cite[Proposition 11.3]{BH1}, the endomorphism algebra of $\pi_\lambda = \ind_{G(k)_{x,+}}^{G(k)} \chi_\lambda$ is isomorphic to the $\chi_\lambda$-spherical Hecke algebra of $G(k)$, and this has basis given by the characteristic functions of the distinct $G(k)_{x,+}$-double cosets that make up $I(G,G(k)_{x,+},\chi_\lambda)$ described above. In the simply-connected case there are two such double cosets so that $\mathrm{End}_{G(k)}(\pi_\lambda)$ is 2-dimensional. By Schur's lemma, $\pi_\lambda$ is a direct sum of two non-isomorphic irreducible supercuspidal representations of $G(k)$. In the adjoint case, $\mathrm{End}_{G(k)}(\pi_\lambda)$ is 4-dimensional. Then $\pi_\lambda$ is either a direct sum of four non-isomorphic irreducible supercuspidals or a direct sum of two isomorphic irreducible supercuspidals. We are in the former case if and only if the $\chi_\lambda$-spherical Hecke algebra $\cH(G,\chi_\lambda)$ is commutative. As $\sigma$ normalises the Iwahori subgroup $G(k)_{x,0}=G(k)_{x,+}$, we find that in $\cH(G,\chi_\lambda)$, writing $K$ for $G(k)_{x,+}$,
    $$\mathbbm{1}_{K\sigma K} * \mathbbm{1}_{K w_{\alpha_n}K} = \mathbbm{1}_{K\sigma w_{\alpha_n} K} = \mathbbm{1}_{K w_{\alpha_n}\sigma K} = \mathbbm{1}_{K w_{\alpha_n}K} * \mathbbm{1}_{K\sigma K}$$
    by \cite[Section 3.1]{IM1}. Similarly, each of the basis elements $\mathbbm{1}_K$, $\mathbbm{1}_{K\sigma K}$, $\mathbbm{1}_{Kw_{\alpha_n} K}$ and $\mathbbm{1}_{K\sigma w_{\alpha_n} K}$ of $\cH(G,\chi_\lambda)$ commute, so $\cH(G,\chi_\lambda)$ is commutative.
\end{proof}

Assume now that $G=\mathrm{Spin}_{2n+1}$ with $n \geq 5$. We will explicitly describe the two irreducible supercuspidal representations above as compactly induced representations. As in Proposition \ref{prop:Bn}, take $y \in \overline \cC$ to be the point with $\alpha_1(y)=\dots =\alpha_{n-1}(y)=\frac{1}{2(n-1)}$ and $\alpha_n(y)=0$. Then 
$$\Delta(y) = \{1-\alpha_0,\alpha_1,\dots,\alpha_{n-1}, \alpha_{n-1}+\alpha_n, \alpha_{n-1}+2\alpha_n\}.$$
Over $\FF_2$, noting that $\GG_m(\FF_2)=1$, the reductive quotient at $y$ has $\FF_2$-points $\GG_y(\FF_2)=\SL_2(\FF_2)$. Let $\lambda'$ denote the $\FF_2$-linear functional on $G(k)_{y,+:++}$ with support $\Delta(y)$. This is not stable for the action of $\GG_y$, as we know from the classification of stable points in \cite{RLYG12} (see Corollary \ref{cor:Bn}). However, it is $\FF_2$-stable as $\GG_y(\FF_2)=\SL_2(\FF_2)$ acts trivially on $\lambda'$. The argument of \cite{RY14} continues to imply that the intertwining set satisfies $I(G,G(k)_{y,+},\chi_{\lambda'}) \subset G(k)_y=G(k)_{y,0}$. This could otherwise be seen using the same argument as Proposition \ref{prop:BnI}, where this time any lift to $N(k)$ of $w_{\alpha_n}$ is in $G(k)_y$. In fact we have 
$$I(G,G(k)_{y,+},\chi_{\lambda'}) = G(k)_y$$
because $\GG_y(\FF_2)$ acts trivially on $\lambda'$. By \cite[Lemma 2.2, Remark 1]{RY14}, the representation $\ind_{G(k)_{y,+}}^{G(k)}\chi_{\lambda'}$ decomposes as a direct sum of three non-isomorphic supercuspidals, indexed by the three irreducible representations of $\SL_2(\FF_2) \cong S_3$. The image of the Iwahori subgroup $G(k)_{x}=G(k)_{x,+} \subset G(k)_{y}$ in $\GG_y(\FF_2)$ is a subgroup of order 2, given by $U_{\alpha_n}/U_{\alpha_n+1}$. The two nontrivial irreducible representations $\rho_1,\rho_2$ of $\SL_2(\FF_2)$ are nontrivial on any such order 2 subgroup. In the notation of Section 3.1, it follows that 
$$
\Hom_{G(k)}\left(\ind_{G(k)_{x}}^{G(k)} \chi_\lambda, \ind_{G(k)_{y}}^{G(k)} (\chi_{\lambda'})_{\rho_i} \right) \neq 0
$$ for $i=1,2$, because $\lambda$ is nontrivial on $U_{\alpha_n}/U_{\alpha_n+1}$. Therefore 
$$\pi_\lambda \cong \ind_{G(k)_y}^{G(k)}(\chi_{\lambda'})_{\rho_1} \oplus \ind_{G(k)_y}^{G(k)}(\chi_{\lambda'})_{\rho_2}.$$

The proof of Proposition \ref{prop:epipelagic} applies here to show that both of these irreducible supercuspidals are epipelagic of depth $\frac{1}{2(n-1)}$.

\iffalse
\begin{rem}
    When considering extensions of $k$, the construction of $\pi_\lambda$ via compact induction from $G(k)_{x,+}$ cannot be extended beyond the case $q=2$, as only in this case is $G(k)_{x,+:1}^\ab$ larger than $G(k)_{x,+:++}$. At the point $y$, one could replace $\lambda'$ by $\FF_q$-linear functionals, but for large $q$ these can no longer be stable.....
\end{rem}
\fi

\subsection{$G=G_2$, $p=3$}

We take $p$ to be 3 and $G$ to be of type $G_2$. Let $x$ be the barycentre of the facet of $\cC$ defined by $\alpha_2=0$, so $x$ has Kac coordinates $10\Rrightarrow 1$. The root system of $\GG_x$ is the sub-root system of $\Phi$ generated by $\{\alpha_2\}$. Therefore, $\GG_x$ is a quotient of $\SL_2 \times \GG_m$ by a finite central subgroup. One may calculate that 
$$V(x) \cong \bigoplus\limits_{\psi \in \cS(x)} V_\psi$$
for $\cS(x) = \Delta(x) \sqcup \{3\alpha_1+\alpha_2, 3\alpha_1+2\alpha_2\}$ and $\Delta(x) = \{1-\alpha_0,1-\alpha_0+\alpha_2,\alpha_1,\alpha_1+\alpha_2\}$. The irreducible subrepresentations of $V(x)$ under the action of $\GG_x$ are $V_{1-\alpha_0} \oplus V_{1-\alpha_0+\alpha_2}$, $V_{\alpha_1}\oplus V_{\alpha_1+\alpha_2}$ and $V_{3\alpha_1+\alpha_2}\oplus V_{3\alpha_1+2\alpha_2}$.

Define an $\FF_q$-linear functional $\lambda \in V(x)^\vee$ as follows. Decompose $\lambda$ as 
$$\lambda = \bigoplus\limits_{\psi \in \cS(x)} \lambda_{-\dot\psi}$$ with $\lambda_{-\dot\psi} \in V_\psi^\vee$. Set $\lambda_{-(\alpha_1+\alpha_2)}=0, \lambda_{-(3\alpha_1+\alpha_2)}=0$. Take $\lambda_{\alpha_0}$, $\lambda_{\alpha_0-\alpha_2}$, $\lambda_{-\alpha_1}$ and $\lambda_{-(3\alpha_1+2\alpha_2)}$ to be any nonzero $\FF_q$-linear functionals in the respective $V_\psi^\vee$. We claim that this gives an $\FF_q$-stable functional.

Let $g \in \GG_x(\FF_q)$. The factor $\GG_m$ of $\SL_2 \times \GG_m$ acts by nonzero scaling on each summand of $V(x)^\vee$. Identify $V_{1-\alpha_0}^\vee \oplus V_{1-\alpha_0+\alpha_2}^\vee$, $V_{\alpha_1}^\vee\oplus V_{\alpha_1+\alpha_2}^\vee$ and $V_{3\alpha_1+\alpha_2}^\vee\oplus V_{3\alpha_1+2\alpha_2}^\vee$ as the standard representation of $\SL_2$. As $(\lambda_{\alpha_0},\lambda_{\alpha_0-\alpha_2})$, $(\lambda_{-\alpha_1}, 0)$ and $(0, \lambda_{-(3\alpha_1+2\alpha_2)})$ are pairwise linearly independent under this identification, so are their images under the action of $g$. Therefore, $\supp(g\cdot \lambda)$ contains at least two elements of each of $\{1-\alpha_0,\alpha_1,3\alpha_1+\alpha_2\}$ and $\{1-\alpha_0+\alpha_2,\alpha_1+\alpha_2,3\alpha_1+2\alpha_2\}$. It also contains at least one element of each of $\{1-\alpha_0,1-\alpha_0+\alpha_2\}$, $\{\alpha_1,\alpha_1+\alpha_2\}$ and $\{3\alpha_1+\alpha_2,3\alpha_1+2\alpha_2\}$. 

Let $$\gamma =  r_1\omega_1^\vee + r_2\omega_2^\vee$$ be a cocharacter of $\TT_x$, where $r_i \in \ZZ$. Suppose that for some $g \in \GG_x(\FF_q)$ we have that 
$$\pair{\dot\psi}{\gamma} \leq 0 \text{ for all $\psi \in \supp(g\cdot\lambda)$.}$$
We have the following inequalities:
\begin{align*}
    \pair{-\alpha_0}{\gamma} \leq 0 &\iff 3r_1+2r_2 \geq 0 \\
    \pair{-\alpha_0+\alpha_2}{\gamma} \leq 0 &\iff 3r_1+r_2 \geq 0 \\
    \pair{\alpha_1}{\gamma} \leq 0 &\iff r_1 \leq 0 \\
    \pair{\alpha_1+\alpha_2}{\gamma} \leq 0 &\iff r_1+r_2 \leq 0 \\
    \pair{3\alpha_1+\alpha_2}{\gamma} \leq 0 &\iff 3r_1+r_2 \leq 0 \\
    \pair{3\alpha_1+2\alpha_2}{\gamma} \leq 0 &\iff 3r_1+2r_2 \leq 0 
\end{align*}

From the above discussion, $\supp(g\cdot \lambda)$ contains one of the following sets:
\begin{equation}\label{eqn:quad}
\{1-\alpha_0,\alpha_1,1-\alpha_0+\alpha_2,3\alpha_1+2\alpha_2\}, \{1-\alpha_0,\alpha_1,\alpha_1+\alpha_2,3\alpha_1+2\alpha_2\},
\end{equation}
$$
\{1-\alpha_0,3\alpha_1+\alpha_2,1-\alpha_0+\alpha_2,\alpha_1+\alpha_2\}, \{1-\alpha_0,3\alpha_1+\alpha_2,3\alpha_1+2\alpha_2,\alpha_1+\alpha_2\},
$$
$$
\{\alpha_1,3\alpha_1+\alpha_2,1-\alpha_0+\alpha_2,\alpha_1+\alpha_2\}, \{\alpha_1,3\alpha_1+\alpha_2,1-\alpha_0+\alpha_2,3\alpha_1+2\alpha_2\}.
$$

In the first two cases, the inequalities force $3r_1+2r_2=0$, $r_1 \leq 0$ and $3r_1+r_2 \geq 0$ or $r_1+r_2 \leq 0$. The only solution is $r_1=r_2=0$. In the last two cases, the inequalities force $3r_1+r_2=0$, $r_1 \leq 0$ and $r_1+r_2 \leq 0$ or $3r_1+2r_2 \leq 0$. The only solution is again $r_1=r_2=0$. In the third case, $3r_1+r_2=0$, $r_1+r_2 \leq 0$ and $3r_1+2r_2 \geq 0$. In the fourth case, $3r_1+2r_2=0$, $r_1+r_2 \leq 0$ and $3r_1+r_2 \leq 0$. Once again, the only solution is $r_1=r_2=0$.

This verifies the $\FF_q$-stability of $\lambda$. By Theorem \ref{thm:main}, fixing a nontrivial additive character $\chi: \FF_q \to \CC^\times$, the compactly induced representation $\pi_\lambda := \ind_{G(k)_{x,+}}^{G(k)} \chi_\lambda$ is a finite direct sum of irreducible supercuspidal representations of $G(k)$.

Let $G(k)_{x,\lambda}$ denote the stabiliser of $\chi_\lambda$ in $G(k)_{x,0}$.

\begin{prop}
    We have $I(G,G(k)_{x,+},\chi_\lambda)=G(k)_{x,\lambda}$.
\end{prop}
\begin{proof}
    Let $n \in N(k)$ with image $w\in W(G)$ and $a,b \in G(k)_{x,0}$ such that $a^{-1}nb \in I(G,G(k)_{x,+},\chi_\lambda)$. It suffices to show that $w(x)=x$, as then this implies $a^{-1}nb \in G(k)_{x,\lambda}$. Any lift $n_{\alpha_2} \in N(k)$ of the reflection in $\alpha_2=0$ lies in $G(k)_{x,0}=G(k)_x$ as it fixes $x$, where this last equality holds because $G$ is simply connected. By replacing $n$ by $nn_{\alpha_2}^{-1}$ and $b$ by $n_{\alpha_2}b$ if necessary, we may assume that $\supp(b\cdot \lambda) \supset \{1-\alpha_0,\alpha_0\}$. Since $w(1-\alpha_0)+w(\alpha_0)=1$ is positive on $x$, at least one of $w(1-\alpha_0)$ and $w(\alpha_0)$ is in $\Psi_x^+$. By Lemma \ref{lem:weylsupport}, this implies that 
    $$\{w(1-\alpha_0),w(\alpha_0)\} = \{1-\alpha_0,\alpha_0\} \text{ or } \{1-\alpha_0+\alpha_2,\alpha_0-\alpha_2=3\alpha_1+\alpha_2\}.$$
    By replacing $n$ by $n_{\alpha_2}n$ and $a$ by $n_{\alpha_2}a$ if necessary, we may assume that $\{w(1-\alpha_0),w(\alpha_0)\} = \{1-\alpha_0,\alpha_0\}$. The support of $b\cdot \lambda$ is larger than $\{1-\alpha_0,\alpha_0\}$, and in fact must also contain one of the pairs 
    $$\{\alpha_1,1-\alpha_0+\alpha_2\}, \{\alpha_1,\alpha_1+\alpha_2\}, \{\alpha_1+\alpha_2, 3\alpha_1+\alpha_2\}$$
    as we have seen in (\ref{eqn:quad}). In each pair, a positive linear combination is $1-\alpha_0$ or $\alpha_0$, and it follows that $w$ must map at least one element $\psi$ of any pair contained in $\supp(b\cdot \lambda)$ to an element of $\Psi_x^+$. By Lemma \ref{lem:weylsupport}, such an element must also be in $\supp(a\cdot \lambda) \subset \cS(x)$. Let $\dot w$ denote the image of $w$ under the projection $W(G) \to W(\Phi)$. In particular, we must have that 
    \begin{equation}\label{eqn:g2weyl}
    \dot w(\dot \psi) \in \{\alpha_1,\alpha_1+\alpha_2,3\alpha_1+\alpha_2,3\alpha_1+\alpha_2,-\alpha_0,-\alpha_0+\alpha_2\}.
    \end{equation}

    Suppose that $w(\alpha_0)=1-\alpha_0$. Then $\dot w$ is either $w_{\alpha_0}$ or the long element of the Weyl group, $w_0=-1$. In the case $\dot w = w_{\alpha_0}$, in order for (\ref{eqn:g2weyl}) to hold, and for $w(\psi) \in \cS(x)$ for such $\psi$, we must have that $\psi=\alpha_1$ and $w(\alpha_1)=\alpha_1$. Together with the equations $w(\alpha_0)=1-\alpha_0$ and $\alpha_0=3\alpha_1+2\alpha_2$, this implies that 
    $$w(\alpha_2) = \frac{1}{2}-3\alpha_1-\alpha_2,$$
    which is impossible. In the case $\dot w = -1$, in order for (\ref{eqn:g2weyl}) to hold, and for $w(\psi) \in \cS(x)$ for such $\psi$, we must have that $w(1-\alpha_0+\alpha_2)=3\alpha_1+\alpha_2=\alpha_0-\alpha_2$. Together with the equations $w(\alpha_0)=1-\alpha_0$ and $\alpha_0=3\alpha_1+2\alpha_2$, this implies that $w(\alpha_2)=-\alpha_2$ and so 
    $$w(\alpha_1) = \frac{1}{3}-\alpha_1,$$
    which is again impossible.

    Suppose that $w(\alpha_0)=\alpha_0$. Then $\dot w$ is either $w_{\alpha_1}$ or the identity element of the Weyl group. In the case $\dot w=w_{\alpha_1}$, (\ref{eqn:g2weyl}) cannot occur. If $\dot w=1$, then by Lemma \ref{lem:weylsupport}, we find that $w$ must fix an element of one of the three pairs of affine roots, so that $w$ is trivial and indeed $w(x)=x$.
\end{proof}

We now describe $G(k)_{x,\lambda}$. Firstly, we claim that $\GG_x \cong \GL_2$. Our argument follows \cite[Section 4.2]{Rom16}. We know that $\GG_x$ is a quotient of $\SL_2 \times \GG_m$ by a finite central subgroup and so is isomorphic to one of $\SL_2 \times \GG_m$, $\mathrm{PGL}_2 \times \GG_m$ or $\GL_2$ by \cite[Lemma 8.1]{GKM04}. The map $\GG_m \to \GG_x$ on the second factor is given by $\omega_1^\vee \in \Hom_{\FF_q}(\GG_m,\TT_x)$. If $\GG_x \cong \SL_2 \times \GG_m$ or $\mathrm{PGL}_2\times \GG_m$, then there would be a character $\beta \in \Hom_{\FF_q}(\TT_x,\GG_m) \cong X^*(T)$ that is trivial on $\SL_2 \cap \TT_x$ or $\mathrm{PGL}_2 \cap \TT_x$ respectively (so $\pair{\beta}{\alpha_2^\vee}=0$), and satisfies $\pair{\beta}{\omega_1^\vee}=1$. But no such character $\beta$ exists. So $\GG_x \cong \GL_2$.

The restrictions of each of $V_{1-\alpha_0}^\vee \oplus V_{1-\alpha_0+\alpha_2}^\vee$, $V_{\alpha_1}^\vee\oplus V_{\alpha_1+\alpha_2}^\vee$ and $V_{3\alpha_1+\alpha_2}^\vee\oplus V_{3\alpha_1+2\alpha_2}^\vee$ to $\SL_2 \leq \GG_x$ are the standard representation of $\SL_2$. The action of the centre of $\GG_x$, given by $\omega_1^\vee(\GG_m)$, is determined by the following equations
$$\pair{\alpha_0}{\omega_1^\vee}=3, \pair{-\alpha_1}{\omega_1^\vee}=-1, \pair{-3\alpha_1-\alpha_2}{\omega_1^\vee}=-3.$$
Identifying $\omega_1^\vee(t)$ with $\begin{psmallmatrix}
    t&0\\0&t
\end{psmallmatrix}\in \GL_2$, we find 
$$V_{1-\alpha_0}^\vee \oplus V_{1-\alpha_0+\alpha_2}^\vee \cong \mathrm{std}\otimes \det$$
$$V_{\alpha_1}^\vee\oplus V_{\alpha_1+\alpha_2}^\vee \cong \mathrm{std} \otimes \det^{-1}$$
and 
$$V_{3\alpha_1+\alpha_2}^\vee\oplus V_{3\alpha_1+2\alpha_2}^\vee \cong \mathrm{std}\otimes \det^{-2}$$
where $\mathrm{std}$ denotes the standard representation of $\GL_2$. Suppose $g \in \GL_2(\FF_q)$ stabilises $\lambda$. We previously set $\lambda_{-(\alpha_1+\alpha_2)}=0$ and $\lambda_{-(3\alpha_1+\alpha_2)}=0$, so $g$ must stabilise $\begin{psmallmatrix}
    1\\0
\end{psmallmatrix} \in \mathrm{std}\otimes \det^{-1}$ and $\begin{psmallmatrix}
    0\\1
\end{psmallmatrix} \in \mathrm{std}\otimes \det^{-2}$. The only such matrix is the identity. We conclude that $G(k)_{x,\lambda}=G(k)_{x,+}$.

\begin{cor}
    Fix a nontrivial additive character $\chi: \FF_q \to \CC^\times$. The compactly induced representation $\pi_\lambda = \ind_{G(k)_{x,+}}^{G(k)} \chi_\lambda$ is an irreducible supercuspidal representation of $G(k)$.
\end{cor}

\begin{prop}
    The representation $\pi_\lambda$ is epipelagic of depth $1/2$.
\end{prop}
\begin{proof}
    By the proof of Proposition \ref{prop:epipelagic}, if $\pi_\lambda$ has depth $r$, then there exists $y \in \cA$ and $g \in \GG_x(\FF_q)$ such that $g\cdot \lambda$ is trivial on $G(k)_{x,+} \cap G(k)_{y,r+}$. For any $g$, by (\ref{eqn:quad}), the support of $g\cdot \lambda$ contains $\{\psi,1-\psi\}$ for one of $\psi=\alpha_0,\alpha_0-\alpha_2$. But for $g\cdot \lambda$ to be trivial on $G(k)_{x,+} \cap G(k)_{y,r+}$, we must then have that $\psi(y) \leq r$ and $(1-\psi)(y) \leq r$ so that $r \geq 1/2$.

    On the other hand, setting $z=\frac{1}{2}(\omega_1^\vee-\omega_2^\vee)$, we have $1-\alpha_0(z)=\alpha_0(z) = 1/2$, $\alpha_1(z)=1/2$ and $1-(3\alpha_1+\alpha_2)(z)=0 \leq 1/2$. Thus $\lambda$ is trivial on $G(k)_{z,\frac{1}{2}+}$ and so $\pi_\lambda$ has a nonzero fixed vector under $G(k)_{z,\frac{1}{2}+}$. Therefore, $\pi_\lambda$ is epipelagic of depth $1/2$.
\end{proof}

\begin{prop}
    The representation $\pi_\lambda$ does not arise from the construction of \cite{RY14}.
\end{prop}
\begin{proof}
    When $G=G_2$, by \cite[Theorem 4.2]{FR17}, the points $y \in \cA$ at which $V_y^\vee$ admits stable vectors are conjugate under the affine Weyl group to $\frac{1}{m} (\omega_1^\vee+\omega_2^\vee)$ for $m=2,3,6$. Under the construction of \cite{RY14}, these stable vectors give rise to epipelagic representations of depth $1/2,1/3$ and $1/6$ respectively, by Proposition \ref{prop:epipelagic}. The point $y\in \cA$ with Kac coordinates $01\Rrightarrow 0$ is conjugate to $\frac{1}{2}(\omega_1^\vee+\omega_2^\vee)$. As $\pi_\lambda$ has depth $1/2$, it suffices to show that $\pi_\lambda$ is not one of the representations $\pi_y(\lambda',\rho)$ where $\lambda' \in V_y^\vee(\FF_q)$ is a stable vector and $\rho \in \mathrm{Irr}(\cH_{\lambda'})$, as defined in Section \ref{sec:epipelagic}.

    Fix a Haar measure $\mu$ on $G(k)$. We have the containments 
    $$\cI_+ \supset G(k)_{x,+}, G(k)_{y,+} \supset \cI_{++}$$ where $\cI \supset \cI_+ \supset \cI_{++}$ is the Moy--Prasad filtration corresponding to the barycentre of $\cC$. The images of $G(k)_{x,+}$ and $G(k)_{y,+}$ in 
    $$\cI_+/\cI_{++} \cong \bigoplus\limits_{\psi \in\Pi} V_\psi$$ are given by $V_{\alpha_1} \oplus V_{1-\alpha_0}$ and $V_{\alpha_2}$. It follows that 
    $$\mu(G(k)_{y,+}) = q^{-1}\mu(G(k)_{x,+}) = q^{-2}\mu(\cI_+).$$
    Recall (for example, \cite[Section 7]{GR10}) that the formal degree with respect to $\mu$ of a compactly induced representation $\ind_J^{G(k)}V$ is $\dim V/\mu(J)$. By Proposition \ref{prop:RY}, if $\pi_\lambda$ is isomorphic to some $\pi(\lambda',\rho)$, then we must have $$\frac{|A_{\lambda'}|}{\dim \chi_{\lambda',\rho}} = q.$$
    In particular, 3 divides $|A_{\lambda'}|$. It suffices to prove that there is no order 3 element of $\GG_y(\overline\FF_q)$ that stabilises a stable vector in $V_y^\vee$. The reductive quotient $\GG_y$ is isogenous to $\SL_2 \times \SL_2$, whose centre is a 2-group, so we may replace $\GG_y$ by $\SL_2 \times \SL_2$. The order 3 elements of $\SL_2(\overline \FF_3) \times \SL_2(\overline \FF_3)$ are all conjugate to one of 
    \begin{equation}\label{eqn:order3}
        \left(\begin{pmatrix}
        1&0\\0&1
    \end{pmatrix}, \begin{pmatrix}
        1&1\\0&1
    \end{pmatrix}\right), \left(\begin{pmatrix}
        1&1\\0&1
    \end{pmatrix}, \begin{pmatrix}
        1&0\\0&1
    \end{pmatrix}\right), \left(\begin{pmatrix}
        1&1\\0&1
    \end{pmatrix}, \begin{pmatrix}
        1&1\\0&1
    \end{pmatrix}\right).
    \end{equation}
    As a representation of $\SL_2(\overline \FF_3) \times \SL_2(\overline \FF_3)$, $V_y^\vee$ is isomorphic to $(P_1 \boxtimes P_3)(\overline \FF_3)$, by \cite[Proposition 4.2]{Rom16}, where for any commutative ring $A$, $P_n(A)$ is the space of homogeneous degree-$n$ polynomials over $A$ in two variables, with $\SL_2(A)$ acting by 
    $$\begin{pmatrix}
        s&t\\u&w
    \end{pmatrix} \cdot f(X,Y) = f(sX+uY,tX+wY).$$
    Take an arbitrary element 
    $$F(X,Y,Z,W) = (aZ+bW)\otimes X^3 + (cZ+dW)\otimes X^2Y + (eZ+fW)\otimes XY^2 + (gZ+hW)\otimes Y^3 \in (P_1 \boxtimes P_3)(\overline \FF_3)$$
    where $a,b,c,d,e,f,g,h \in \overline \FF_3$. If $F$ is stabilised by one of the elements of (\ref{eqn:order3}), one checks that $d=f=0$ or $e=f=0$. By \cite[Proposition 2.15]{Rom16}, there is a polynomial $\Delta(F)$ in $a,\dots,h$ such that $F$ is stable for the action of $\SL_2(\overline \FF_3) \times \SL_2(\overline \FF_3)$ if and only if $\Delta(F)\neq 0$. Explicitly, there are polynomials $H_6(F)$ and $G_6(F)$ such that $\Delta(F)=H_6(F)^3G_6(F)$ and $H_6(F)$ reduces in characteristic 3 to $c^3f^3-d^3e^3$. We see that this is 0 for any $F$ stabilised by an element of order 3, so that such $F$ is not a stable vector.

\end{proof}
% using Jordan normal form

\begin{rem}
    One is naturally interested in describing the corresponding Langlands parameter $\varphi$ to $\pi_\lambda$. As the dual reductive group to $G_2$ is again $G_2$, $\varphi$ is a continuous homomorphism $\varphi: \cW_k \times \SL_2(\CC) \to G_2(\CC)$, where $\cW_k$ denotes the Weil group of $k$. The formal degree conjecture as reformulated in \cite[Conjecture 7.1]{GR10} (and originally due to \cite{HII08}) predicts that the formal degree of $\pi_\lambda$ matches numerical invariants attached to $\varphi$. In this remark we deduce some facts about $\varphi$, assuming the formal degree conjecture. For more details, see the forthcoming thesis of the author.

    Let $\pi_s$ denote a simple supercuspidal for $G_2(k)$ as constructed by \cite{GR10}, coming from a choice $\chi$ of `affine generic character', and let $\varphi_s$ be the corresponding parameter under the formal degree conjecture. In \cite[Section 9.5]{GR10}, Gross--Reeder show that $\varphi_s$ is a simple wild parameter, meaning that $\varphi_s$ is trivial on the $\SL_2$ factor, it has minimal adjoint Swan conductor given by the rank of $G_2$ (so $b(\varphi_s)=2$), and the image of the inertia subgroup of $\cW_k$ has finite centraliser in $G_2(\CC)$. Their argument relies on a compatibility assumption of unramified base change of the simple supercuspidals and restriction of the Langlands parameters. More precisely, if $k_m$ denotes the degree $m$ unramified extension of $k$, one can pull back $\chi$ under the trace map $\FF_{q^m} \to \FF_q$ to an affine generic character $\chi_m$ for $G_2(k_m)$, and hence produce a simple supercuspidal $\pi_{s,m}$ of $G_2(k_m)$. They make the assumption that the corresponding parameter to this pullback is the restriction of $\varphi_s$ to $\cW_{k_m}$.

    In our setting, we may similarly pullback an $\FF_q$-stable functional $\lambda$ to an $\FF_{q^m}$-stable functional $\lambda_m$ on $G_2(k_m)_{x,+:1}^\ab$ using the trace map $\FF_{q^m} \to \FF_q$. More precisely, $G_2(k_m)_{x,+:1}^\ab \cong V(x) \otimes_{\FF_q}\FF_{q^m}$ is a direct sum of distinguished lines $V_\psi \otimes_{\FF_q} \FF_{q^m}$ over $\psi \in \cS(x)$, and we take the sum of trace maps $V_\psi \otimes_{\FF_q} \FF_{q^m}\to V_\psi$. As the trace map on finite fields is nonzero, the support of $\lambda_m$ is the same as the support of $\lambda$. We therefore produce irreducible supercuspidals $\pi_{\lambda_m}$ where $\lambda_1=\lambda$. Replacing $\pi_{s,m}$ by $\pi_{\lambda_m}$ has the effect of multiplying the formal degree of $\pi_{s,m}$ by $q^m$. This comes from the fact that $G(k)_{x,+}$ is index $q$ in $\cI_+$ and that $G_2$ has trivial centre. We may then apply the same argument as in \cite[Section 9.5]{GR10} to deduce, under the analogous assumption that the corresponding parameter to $\pi_{\lambda_m}$ is the restriction of $\varphi$ to $\cW_{k_m}$, that $\varphi$ is trivial on the $\SL_2$ factor, the image of inertia has finite centraliser in $G_2(\CC)$, and the Swan conductor of $\varphi$ is given by $b(\varphi)=b(\varphi_s)+2=4$.
\end{rem}

\newpage

\bibliography{paper_references.bib}

@article{BT2,
author = {Bruhat, F. and Tits, J.},
journal = {Inst. Hautes \'{E}tudes Sci. Publ. Math.},
pages = {197-376},
title = {Groupes r\'{e}ductifs sur un corps local. II. Sch'{e}mas en groupes. Existence d'une donn\'{e}e radicielle valu\'{e}e},
volume = {60},
year = {1984},
}

@article {Bus90,
    AUTHOR = {Bushnell, Colin J.},
     TITLE = {Induced representations of locally profinite groups},
   JOURNAL = {J. Algebra},
  FJOURNAL = {Journal of Algebra},
    VOLUME = {134},
      YEAR = {1990},
    NUMBER = {1},
     PAGES = {104--114},
      ISSN = {0021-8693,1090-266X},
   MRCLASS = {22E50 (22D30)},
  MRNUMBER = {1068417},
MRREVIEWER = {Joe\ Repka},
       DOI = {10.1016/0021-8693(90)90213-8},
       URL = {https://doi.org/10.1016/0021-8693(90)90213-8},
}

@book{BH1,
      author        = "Bushnell, C. J. and Henniart, G.",
      title         = "{The Local Langlands Conjecture for GL(2)}",
      publisher     = "Springer Berlin",
      series        = "Grundlehren der mathematischen Wissenschaften",
      year          = "2006",
}

@book {BK93,
    AUTHOR = {Bushnell, Colin J. and Kutzko, Philip C.},
     TITLE = {The admissible dual of {${\rm GL}(N)$} via compact open
              subgroups},
    SERIES = {Annals of Mathematics Studies},
    VOLUME = {129},
 PUBLISHER = {Princeton University Press, Princeton, NJ},
      YEAR = {1993},
     PAGES = {xii+313},
      ISBN = {0-691-03256-4; 0-691-02114-7},
   MRCLASS = {22E50 (22-02)},
  MRNUMBER = {1204652},
MRREVIEWER = {Mark\ Reeder},
       DOI = {10.1515/9781400882496},
       URL = {https://doi.org/10.1515/9781400882496},
}

@article {BK94,
    AUTHOR = {Bushnell, Colin J. and Kutzko, Philip C.},
     TITLE = {The admissible dual of {${\rm SL}(N)$}. {II}},
   JOURNAL = {Proc. London Math. Soc. (3)},
  FJOURNAL = {Proceedings of the London Mathematical Society. Third Series},
    VOLUME = {68},
      YEAR = {1994},
    NUMBER = {2},
     PAGES = {317--379},
      ISSN = {0024-6115,1460-244X},
   MRCLASS = {22E50 (22E35)},
  MRNUMBER = {1253507},
MRREVIEWER = {Leticia\ Barchini},
       DOI = {10.1112/plms/s3-68.2.317},
       URL = {https://doi.org/10.1112/plms/s3-68.2.317},
}

@book{Car1,
      author        = "Carter, R. W.",
      title         = "{Finite Groups of Lie Type}",
      publisher     = "John Wiley and Sons",
      series        = "Pure and Applied Mathematics",
      year          = "1985",
}

@article {Dav54,
    AUTHOR = {Davis, Chandler},
     TITLE = {Theory of positive linear dependence},
   JOURNAL = {Amer. J. Math.},
  FJOURNAL = {American Journal of Mathematics},
    VOLUME = {76},
      YEAR = {1954},
     PAGES = {733--746},
      ISSN = {0002-9327,1080-6377},
   MRCLASS = {09.0X},
  MRNUMBER = {64011},
MRREVIEWER = {L.\ M.\ Blumenthal},
       DOI = {10.2307/2372648},
       URL = {https://doi.org/10.2307/2372648},
}

@article {Fin21,
    AUTHOR = {Fintzen, Jessica},
     TITLE = {Types for tame {$p$}-adic groups},
   JOURNAL = {Ann. of Math. (2)},
  FJOURNAL = {Annals of Mathematics. Second Series},
    VOLUME = {193},
      YEAR = {2021},
    NUMBER = {1},
     PAGES = {303--346},
      ISSN = {0003-486X,1939-8980},
   MRCLASS = {22E50},
  MRNUMBER = {4199732},
MRREVIEWER = {Corina\ Ciobotaru},
       DOI = {10.4007/annals.2021.193.1.4},
       URL = {https://doi.org/10.4007/annals.2021.193.1.4},
}

@article {FR17,
    AUTHOR = {Fintzen, Jessica and Romano, Beth},
     TITLE = {Stable vectors in {M}oy-{P}rasad filtrations},
   JOURNAL = {Compos. Math.},
  FJOURNAL = {Compositio Mathematica},
    VOLUME = {153},
      YEAR = {2017},
    NUMBER = {2},
     PAGES = {358--372},
      ISSN = {0010-437X,1570-5846},
   MRCLASS = {22E50 (11S37 14L24 20G25)},
  MRNUMBER = {3705228},
MRREVIEWER = {U.\ K.\ Anandavardhanan},
       DOI = {10.1112/S0010437X16008228},
       URL = {https://doi.org/10.1112/S0010437X16008228},
}

@misc{FS25,
      title={Construction of tame supercuspidal representations in arbitrary residue characteristic}, 
      author={Jessica Fintzen and David Schwein},
      year={2025},
      eprint={2501.18553},
      archivePrefix={arXiv},
      primaryClass={math.RT},
      url={https://arxiv.org/abs/2501.18553},
      note = {arXiv:2501.18553} 
}

@misc{Gas20,
      title={Shallow Characters and Supercuspidal Representations}, 
      author={Stella Sue Gastineau},
      year={2020},
      eprint={2011.00049},
      archivePrefix={arXiv},
      primaryClass={math.RT},
      url={https://arxiv.org/abs/2011.00049}, 
      note = {arXiv:2011.00049}
}

@article {GKM04,
    AUTHOR = {Goresky, Mark and Kottwitz, Robert and Macpherson, Robert},
     TITLE = {Homology of affine {S}pringer fibers in the unramified case},
   JOURNAL = {Duke Math. J.},
  FJOURNAL = {Duke Mathematical Journal},
    VOLUME = {121},
      YEAR = {2004},
    NUMBER = {3},
     PAGES = {509--561},
      ISSN = {0012-7094,1547-7398},
   MRCLASS = {14M15 (14L30)},
  MRNUMBER = {2040285},
MRREVIEWER = {Shrawan\ Kumar},
       DOI = {10.1215/S0012-7094-04-12135-9},
       URL = {https://doi.org/10.1215/S0012-7094-04-12135-9},
}

@article {GR10,
    AUTHOR = {Gross, Benedict H. and Reeder, Mark},
     TITLE = {Arithmetic invariants of discrete {L}anglands parameters},
   JOURNAL = {Duke Math. J.},
  FJOURNAL = {Duke Mathematical Journal},
    VOLUME = {154},
      YEAR = {2010},
    NUMBER = {3},
     PAGES = {431--508},
      ISSN = {0012-7094,1547-7398},
   MRCLASS = {11S37 (11S15 22E50)},
  MRNUMBER = {2730575},
MRREVIEWER = {Joshua\ M.\ Lansky},
       DOI = {10.1215/00127094-2010-043},
       URL = {https://doi.org/10.1215/00127094-2010-043},
}

@article {HII08,
    AUTHOR = {Hiraga, Kaoru and Ichino, Atsushi and Ikeda, Tamotsu},
     TITLE = {Formal degrees and adjoint {$\gamma$}-factors},
   JOURNAL = {J. Amer. Math. Soc.},
  FJOURNAL = {Journal of the American Mathematical Society},
    VOLUME = {21},
      YEAR = {2008},
    NUMBER = {1},
     PAGES = {283--304},
      ISSN = {0894-0347,1088-6834},
   MRCLASS = {22E50},
  MRNUMBER = {2350057},
       DOI = {10.1090/S0894-0347-07-00567-X},
       URL = {https://doi.org/10.1090/S0894-0347-07-00567-X},
}

@article {IM1,
    AUTHOR = {Iwahori, N. and Matsumoto, H.},
     TITLE = {On some {B}ruhat decomposition and the structure of the
              {H}ecke rings of {${\mathfrak p}$}-adic {C}hevalley groups},
   JOURNAL = {Inst. Hautes \'Etudes Sci. Publ. Math.},
  FJOURNAL = {Institut des Hautes \'Etudes Scientifiques. Publications
              Math\'ematiques},
    NUMBER = {25},
      YEAR = {1965},
     PAGES = {5--48},
      ISSN = {0073-8301,1618-1913},
   MRCLASS = {20.70 (14.50)},
  MRNUMBER = {185016},
MRREVIEWER = {Rimhak\ Ree},
       URL = {http://www.numdam.org/item?id=PMIHES_1965__25__5_0},
}

@article {Kim07,
    AUTHOR = {Kim, Ju-Lee},
     TITLE = {Supercuspidal representations: an exhaustion theorem},
   JOURNAL = {J. Amer. Math. Soc.},
  FJOURNAL = {Journal of the American Mathematical Society},
    VOLUME = {20},
      YEAR = {2007},
    NUMBER = {2},
     PAGES = {273--320},
      ISSN = {0894-0347,1088-6834},
   MRCLASS = {22E50 (20G25 22E35)},
  MRNUMBER = {2276772},
MRREVIEWER = {U.\ K.\ Anandavardhanan},
       DOI = {10.1090/S0894-0347-06-00544-3},
       URL = {https://doi.org/10.1090/S0894-0347-06-00544-3},
}

@article{Kut77,
 ISSN = {00029939, 10886826},
 URL = {http://www.jstor.org/stable/2041005},
 author = {P. C. Kutzko},
 journal = {Proceedings of the American Mathematical Society},
 number = {1},
 pages = {173--175},
 publisher = {American Mathematical Society},
 title = {Mackey's Theorem for Nonunitary Representations},
 urldate = {2025-01-28},
 volume = {64},
 year = {1977}
}

@article {Mor99,
    AUTHOR = {Morris, Lawrence},
     TITLE = {Level zero {$\bf G$}-types},
   JOURNAL = {Compositio Math.},
  FJOURNAL = {Compositio Mathematica},
    VOLUME = {118},
      YEAR = {1999},
    NUMBER = {2},
     PAGES = {135--157},
      ISSN = {0010-437X,1570-5846},
   MRCLASS = {22E50},
  MRNUMBER = {1713308},
MRREVIEWER = {Goran\ Mui\'c},
       DOI = {10.1023/A:1001019027614},
       URL = {https://doi.org/10.1023/A:1001019027614},
}

@article {Mum77,
    AUTHOR = {Mumford, David},
     TITLE = {Stability of projective varieties},
   JOURNAL = {Enseign. Math. (2)},
  FJOURNAL = {L'Enseignement Math\'ematique. Revue Internationale. 2e
              S\'erie},
    VOLUME = {23},
      YEAR = {1977},
    NUMBER = {1-2},
     PAGES = {39--110},
      ISSN = {0013-8584},
   MRCLASS = {14D20},
  MRNUMBER = {450272},
MRREVIEWER = {P.\ E.\ Newstead},
}

@article{MP1,
author = {Moy, A. and Prasad, G.},
journal = {Invent. Math.},
pages = {393-408},
title = "{Unrefined minimal $K$-types for $p$-adic groups}",
volume = {116},
year = {1994},
}

@article{MP2,
author = {Moy, A. and Prasad, G.},
journal = {Comment. Math. Helv.},
pages = {98-121},
title = "{Jacquet functors and unrefined minimal $K$-types}",
volume = {71},
year = {1996},
}

@article {RLYG12,
    AUTHOR = {Reeder, Mark and Levy, Paul and Yu, Jiu-Kang and Gross,
              Benedict H.},
     TITLE = {Gradings of positive rank on simple {L}ie algebras},
   JOURNAL = {Transform. Groups},
  FJOURNAL = {Transformation Groups},
    VOLUME = {17},
      YEAR = {2012},
    NUMBER = {4},
     PAGES = {1123--1190},
      ISSN = {1083-4362,1531-586X},
   MRCLASS = {17B20 (17B70)},
  MRNUMBER = {3000483},
MRREVIEWER = {Rutwig\ Campoamor-Stursberg},
       DOI = {10.1007/s00031-012-9196-3},
       URL = {https://doi.org/10.1007/s00031-012-9196-3},
}

@book {Rom16,
    AUTHOR = {Romano, Beth},
     TITLE = {On the local {L}anglands correspondence: {N}ew examples from
              the epipelagic zone},
      NOTE = {Thesis (Ph.D.)--Boston College},
 PUBLISHER = {ProQuest LLC, Ann Arbor, MI},
      YEAR = {2016},
     PAGES = {95},
      ISBN = {978-1339-68032-3},
   MRCLASS = {99-05},
  MRNUMBER = {3517860},
       URL =
              {http://gateway.proquest.com/openurl?url_ver=Z39.88-2004&rft_val_fmt=info:ofi/fmt:kev:mtx:dissertation&res_dat=xri:pqm&rft_dat=xri:pqdiss:10104498},
}

@article {RY14,
    AUTHOR = {Reeder, Mark and Yu, Jiu-Kang},
     TITLE = {Epipelagic representations and invariant theory},
   JOURNAL = {J. Amer. Math. Soc.},
  FJOURNAL = {Journal of the American Mathematical Society},
    VOLUME = {27},
      YEAR = {2014},
    NUMBER = {2},
     PAGES = {437--477},
      ISSN = {0894-0347,1088-6834},
   MRCLASS = {22E50 (11S15)},
  MRNUMBER = {3164986},
MRREVIEWER = {Marko\ Tadi\'c},
       DOI = {10.1090/S0894-0347-2013-00780-8},
       URL = {https://doi.org/10.1090/S0894-0347-2013-00780-8},
}

@article {SS08,
    AUTHOR = {S\'echerre, V. and Stevens, S.},
     TITLE = {Repr\'esentations lisses de {${\rm GL}_m(D)$}. {IV}.
              {R}epr\'esentations supercuspidales},
   JOURNAL = {J. Inst. Math. Jussieu},
  FJOURNAL = {Journal of the Institute of Mathematics of Jussieu. JIMJ.
              Journal de l'Institut de Math\'ematiques de Jussieu},
    VOLUME = {7},
      YEAR = {2008},
    NUMBER = {3},
     PAGES = {527--574},
      ISSN = {1474-7480,1475-3030},
   MRCLASS = {22E50},
  MRNUMBER = {2427423},
       DOI = {10.1017/S1474748008000078},
       URL = {https://doi.org/10.1017/S1474748008000078},
}

@book {Ste16,
    AUTHOR = {Steinberg, Robert},
     TITLE = {Lectures on {C}hevalley groups},
    SERIES = {University Lecture Series},
    VOLUME = {66},
   EDITION = {corrected},
      NOTE = {Notes prepared by John Faulkner and Robert Wilson,
              With a foreword by Robert R. Snapp},
 PUBLISHER = {American Mathematical Society, Providence, RI},
      YEAR = {2016},
     PAGES = {xi+160},
      ISBN = {978-1-4704-3105-1},
   MRCLASS = {20G15 (51F15)},
  MRNUMBER = {3616493},
       DOI = {10.1090/ulect/066},
       URL = {https://doi.org/10.1090/ulect/066},
}

@article {Ste08,
    AUTHOR = {Stevens, Shaun},
     TITLE = {The supercuspidal representations of {$p$}-adic classical
              groups},
   JOURNAL = {Invent. Math.},
  FJOURNAL = {Inventiones Mathematicae},
    VOLUME = {172},
      YEAR = {2008},
    NUMBER = {2},
     PAGES = {289--352},
      ISSN = {0020-9910,1432-1297},
   MRCLASS = {22E50},
  MRNUMBER = {2390287},
MRREVIEWER = {Anne-Marie\ H.\ Aubert},
       DOI = {10.1007/s00222-007-0099-1},
       URL = {https://doi.org/10.1007/s00222-007-0099-1},
}

@incollection {Tit79,
    AUTHOR = {Tits, J.},
     TITLE = {Reductive groups over local fields},
 BOOKTITLE = {Automorphic forms, representations and {$L$}-functions
              ({P}roc. {S}ympos. {P}ure {M}ath., {O}regon {S}tate {U}niv.,
              {C}orvallis, {O}re., 1977), {P}art 1},
    SERIES = {Proc. Sympos. Pure Math.},
    VOLUME = {XXXIII},
     PAGES = {29--69},
 PUBLISHER = {Amer. Math. Soc., Providence, RI},
      YEAR = {1979},
      ISBN = {0-8218-1435-4},
   MRCLASS = {20G25 (20G10)},
  MRNUMBER = {546588},
}

@article {Yu01,
    AUTHOR = {Yu, Jiu-Kang},
     TITLE = {Construction of tame supercuspidal representations},
   JOURNAL = {J. Amer. Math. Soc.},
  FJOURNAL = {Journal of the American Mathematical Society},
    VOLUME = {14},
      YEAR = {2001},
    NUMBER = {3},
     PAGES = {579--622},
      ISSN = {0894-0347,1088-6834},
   MRCLASS = {22E50},
  MRNUMBER = {1824988},
MRREVIEWER = {Bertrand\ Lemaire},
       DOI = {10.1090/S0894-0347-01-00363-0},
       URL = {https://doi.org/10.1090/S0894-0347-01-00363-0},
}
\bibliographystyle{amsalpha}

\end{document}